\documentclass[10pt]{amsart}

\usepackage{latexsym}
\usepackage{amsfonts}
\usepackage{amssymb}
\usepackage{amsmath}
\usepackage{mathrsfs}
\usepackage{hyperref}
\usepackage{tikz}
\usetikzlibrary{matrix,arrows}

\title[Multiplicativity of Connes' calculus]{Multiplicativity of Connes' calculus}
\author{Partha Sarathi Chakraborty}
\address{The Institute of Mathematical Sciences, CIT Campus, Taramani, Chennai
600113}
\email{parthac@imsc.res.in}

\author{Satyajit Guin}
\address{The Institute of Mathematical Sciences, CIT Campus, Taramani, Chennai
600113}
\email{gsatyajit@imsc.res.in}
\date{\today}

\newtheorem{definition}{Definition}[section]
\newtheorem{theorem}[definition]{Theorem}
\newtheorem{lemma}[definition]{Lemma}
\newtheorem{proposition}[definition]{Proposition}
\newtheorem{corollary}[definition]{Corollary}

\newtheorem{remark}[definition]{Remark}

\begin{document}
\begin{abstract}
We consider the quadruples $\,(\mathcal{A},\mathbb{V},D,\gamma)$ where $\mathcal{A}$ is a unital, associative $\mathbb{K}\,$-algebra represented on the $\mathbb{K}\,$-vector space $\mathbb{V}$, $D\in \mathcal{E}nd(\mathbb{V})$,
$\gamma\in\mathcal{E}nd(\mathbb{V})$ is a $\mathbb{Z}_2$-grading operator which commutes with $\mathcal{A}$ and anticommutes with $D$. We prove that the collection of such quadruples, denoted by $\,\widetilde{\mathcal{S}pec}\,$,
is a monoidal category. We consider the monoidal subcategory $\,\widetilde{\mathcal{S}pec_{sub}}\,$ of objects of $\,\widetilde{\mathcal{S}pec}\,$ for which $\gamma\in\pi(\mathcal{A})$. We show that there is a covariant
functor $\,\mathcal{G}:\widetilde{\mathcal{S}pec}\longrightarrow\widetilde{\mathcal{S}pec_{sub}}\,$. Let $\,\Omega_D^\bullet\,$ be the differential graded algebra defined by Connes (\cite{3}) and $DGA$ denotes the category of
differential graded algebras over the field $\mathbb{K}\,$. We show that $\mathcal{F}:\widetilde{\mathcal{S}pec_{sub}}\longrightarrow DGA\,$, given by $(\mathcal{A},\mathbb{V},D,\gamma)\longmapsto\Omega_D^\bullet(\mathcal{A})$,
is a monoidal functor. To show that $\,\mathcal{F}\circ\mathcal{G}\,$ is not trivial we explicitly compute it for the cases of compact manifold and the noncommutative torus along with the associated cohomologies.
\end{abstract}
\maketitle

\section{Introduction}
A noncommutative differential structure on an associative algebra $\mathcal{A}$ over a field $\mathbb{K}$ is the specification of a differential graded algebra(dga), which is interpreted as the space of differential forms.
Study of differential calculus in noncommutative geometry appears in early $80$'s through the invention of noncommutative defferential geometry (\cite{3.0}), and to search for its examples (\cite{10}),(\cite{11}).
Since then quite a lot of works have been done involving differential calculus in various noncommutative contexts for e.g. (\cite{7.3}), (\cite{2}), (\cite{2.1}), (\cite{7.1}), (\cite{7.2}), (\cite{1.1}) and references therein.
In his spectral formulation of the subject, Connes unified various treatments in noncommutative geometry in terms of a $\mathcal{K}$-cycle $(\mathcal{A},\mathcal{H},D)$. He defined a canonical dga $\Omega_D^\bullet(\mathcal{A})$ associated
to a $\mathcal{K}$-cycle $(\mathcal{A},\mathcal{H},D)$ and extended several classical notions including connection, curvature, Yang-Mills action functional etc. to the noncommutative framework. It is also shown in (\cite{3})
that using this dga one can produce Hochschild cocycle and cyclic cocycle (under certain assumption) for Poincar\'{e} dual algebras which establishes $\,\Omega_D^\bullet\,$ worth studying. Since it is possible to multiply
{\it even} $\mathcal{K}$-cycles $$(\mathcal{A}_1,\mathcal{H}_1,D_1,\gamma_1)\otimes(\mathcal{A}_2,\mathcal{H}_2,D_2,\gamma_2):=(\mathcal{A}_1\otimes \mathcal{A}_2,\mathcal{H}_1\otimes \mathcal{H}_1,D_1\otimes 1+\gamma_1
\otimes D_2,\gamma_1\otimes \gamma_2),$$ natural question strikes regarding the behaviour of $\,\Omega_D^\bullet\,$ under this multiplication (\cite{7.0}) and this is the content of this paper. This question was investigated earlier in
(\cite{5}) and main out-turn was that $\Omega_D^\bullet(\mathcal{A}_1\otimes\mathcal{A}_2)\ncong \Omega_{D_1}^\bullet(\mathcal{A}_1)\otimes \Omega_{D_2}^\bullet(\mathcal{A}_2)$ in general. Here $\,\Omega_{D_1}^\bullet
\otimes\Omega_{D_2}^\bullet$ denotes the tensor product of two differential graded algebras. We call it multiplicativity property of $\Omega_D^\bullet$ and hence, in this language, the result in (\cite{5}) states that
$\,\Omega_D^\bullet\,$ is in general not multiplicative.

Since the output of (\cite{5}) is not conclusive we reinvestigate this question. To define $\,\Omega_D^\bullet\,$ one does not use self-adjointness and compactness of the resolvent of $D$. We cast Connes definition
in a slightly more general algebraic framework. We consider the quadruple $(\mathcal{A},\mathbb{V},D,\gamma)$ where $\mathcal{A}$ is an associative, unital algebra over $\mathbb{K}$, represented on a vector space
$\mathbb{V}$, $D\in \mathcal{E}nd(\mathbb{V})$, $\gamma\in \mathcal{E}nd(\mathbb{V})$ is a $\mathbb{Z}_2$-grading operator which commutes with $\mathcal{A}$ and anticommutes with $D$. We show that the collection of such
quadruple $(\mathcal{A},\mathbb{V},D,\gamma)$ is a monoidal category and denote it by $\widetilde{\mathcal{S}pec}\,$. We identify a smaller subcategory $\widetilde{\mathcal{S}pec_{sub}}\,$ and show that there is a covariant
functor $\,\mathcal{G}:\widetilde{\mathcal{S}pec}\longrightarrow\widetilde{\mathcal{S}pec_{sub}}$. Moreover, $\widetilde{\mathcal{S}pec_{sub}}\,$ becomes a monoidal subcategory of $\widetilde{\mathcal{S}pec}\,$. Next we
consider the category $DGA$ of differential graded algebras over a field $\mathbb{K}\,$ and show that the association $\mathcal{F}:(\mathcal{A},\mathbb{V},D,\gamma)\longmapsto\Omega_D^\bullet(\mathcal{A})$ gives a covariant
functor from $\widetilde{\mathcal{S}pec}$ to $DGA$. In this category theoretic language, article (\cite{5}) says that this functor is in general not monoidal. We show that restricted to $\widetilde{\mathcal{S}pec_{sub}}\,,
\,\mathcal{F}$ becomes a monoidal functor. To validate the nontriviality of this functor, i,e. the associated dga $\Omega_D^\bullet$ is not trivial, we explicitly compute $\,\mathcal{F}\circ\mathcal{G}\,$ for the cases of
compact manifold and the noncommutative torus. We also compute the associated cohomologies in each case and it turns out that the resulting dga $\,\Omega_D^\bullet\,$ in these two cases is cohomologically also not trivial.

Organization of this paper is as follows. In section (2) we first define algebraic spectral triple and go through the definition of Connes' calculus $\Omega_D^\bullet\,$. Then we formulate the category $\widetilde{Spec}$
and prove that it is a monoidal category and $\mathcal{F}:(\mathcal{A},\mathbb{V},D,\gamma)\longmapsto\Omega_D^\bullet(\mathcal{A})$ is a covariant functor. Next we identify the subcategory $\widetilde{\mathcal{S}pec_{sub}}
\,$ and obtain the covariant functor $\,\mathcal{G}:\widetilde{\mathcal{S}pec}\longrightarrow\widetilde{\mathcal{S}pec_{sub}}$. Finally we show that the functor $\mathcal{F}$ restricted to $\widetilde{\mathcal{S}pec_{sub}}$
is a monoidal functor between the monoidal categoris $\widetilde{\mathcal{S}pec_{sub}}$ and $DGA$. Sections (3) and (4) have been devoted to the computation for the cases of compact manifold and the noncommutative torus
respectively.

\medskip

\section{Multiplicativity of Connes' Calculus}

\begin{definition}
An algebraic {\it spectral triple} $(\mathcal{A},\mathbb{V},D)$, over an unital associative $\mathbb{K}$-algebra $\mathcal{A}$, consists of the following things~:
\begin{enumerate}
\item a representation $\pi$ of $\mathcal{A}$ on a $\mathbb{K}$-vector space $\mathbb{V}$,
\item a linear operator $D$ acting on $\mathbb{V}$.
\end{enumerate}
\end{definition}
It is said to be an {\it even algebraic spectral triple} if there exists a $\mathbb{Z}_2$-grading $\gamma \in \mathcal{E}nd(\mathbb{V})$ such that $\gamma$ commutes with each element of $\mathcal{A}$ and anticommutes with $D$.
It will be assumed that $\mathcal{A}$ is unital and the unit $1\in\mathcal{A}$ acts as the identity operator on $\mathbb{V}$.

\begin{definition}
Let $\,\Omega^\bullet(\mathcal{A}) = \bigoplus_{k=0}^\infty \Omega^k(\mathcal{A})\,$ be the reduced universal differential graded algebra over $\mathcal{A}\,$. Here 
$\,\Omega^k(\mathcal{A}):=\mathcal{A}\otimes {\bar{\mathcal{A}}}^k\, $, $\, \bar{\mathcal{A}}=\mathcal{A}/\mathbb{K}\, $. The graded product is given by
\begin{eqnarray*}
&    & \left(\sum_k a_{0k}\otimes \overline{a_{1k}}\otimes \ldots \otimes \overline{a_{mk}}\right) . \left(\sum_{k^\prime}b_{0k^\prime}\otimes \overline{b_{1k^\prime}}\otimes \ldots \otimes \overline{b_{nk^\prime}}\right)\\
& := & \sum_{k,k^\prime} a_{0k}\otimes (\otimes_{j=1}^{m-1}\overline{a_{jk}})\otimes \overline{a_{mk}b_{0k^\prime}}\otimes (\otimes_{i=1}^n\overline{b_{ik^\prime}})\\
&    & + \sum_{i=1}^{m-1}(-1)^i a_{0k}\otimes \overline{a_{1k}}\otimes \ldots \otimes \overline{a_{m-i,k}a_{m-i+1,k}}\otimes\ldots \otimes \overline{a_{mk}}\otimes (\otimes_{i=0}^n\overline{b_{ik^\prime}})\\
&    & + (-1)^ma_{0k}a_{1k}\otimes (\otimes_{j=2}^m\overline{a_{jk}})\otimes (\otimes_{i=0}^n\overline{b_{ik^\prime}})\, .
\end{eqnarray*}
for $\,\sum_k a_{0k}\otimes \overline{a_{1k}}\otimes \ldots \otimes \overline{a_{mk}}\in \Omega^m(\mathcal{A})$ and $\,\sum_{k^\prime}b_{0k^\prime}\otimes \overline{b_{1k^\prime}}\otimes \ldots \otimes \overline{b_{nk^\prime}}
\in \Omega^n(\mathcal{A})$. There is a differential $d$ acting on $\Omega^\bullet(\mathcal{A})$ given by
\begin{center}
 $d(a_0\otimes \bar{a_1}\otimes \ldots \otimes \bar{a_k}) = 1\otimes \bar{a_0}\otimes \bar{a_1}\otimes \ldots \otimes \bar{a_k}\, \, \, \forall\, a_j \in \mathcal{A}\, ,$
\end{center}
and it satisfies the relations
\begin{enumerate}
 \item $d^2 \omega = 0\, \, \, \forall\, \omega \in \Omega^\bullet(\mathcal{A})$,
 \item $d(\omega_1\omega_2) = (d\omega_1)\omega_2 + (-1)^{deg (\omega_1)} \omega_1 d\omega_2 \, ,\, \, \forall\, \omega_j \in \Omega^\bullet(\mathcal{A})$.
\end{enumerate}
We can represent $\,\Omega^\bullet(\mathcal{A})$ on $\mathbb{V}$ by ,
\begin{eqnarray*}
\pi(a_0\otimes \overline{a_1}\otimes \ldots \otimes \overline{a_k}) = a_0[D,a_1]\ldots [D,a_k]\quad;\,\, a_j \in \mathcal{A}\,.
\end{eqnarray*}
Let $\,J_0^{(k)} = \{ \omega \in \Omega^k : \pi(\omega) = 0 \}$ and $J^\prime = \bigoplus J_0^{(k)}$. But $J^\prime$ is not a differential graded ideal. We consider $J^\bullet = \bigoplus J^{(k)}$ where 
$J^{(k)} = J_0^{(k)} + dJ_0^{(k-1)}$. Then $J^\bullet$ becomes a differential graded two-sided ideal and hence the quotient $\Omega_D^\bullet = \Omega^\bullet/J^\bullet$ becomes a differential graded algebra.
The representation $\pi$ gives an isomorphism,
\begin{eqnarray}\label{main iso}
\Omega_D^k \cong \pi(\Omega^k)/\pi(dJ_0^{k-1})\quad \forall\,k\geq 1\, .
\end{eqnarray}
\end{definition}

The abstract differential $\,d\,$ induces a differential $\,d\,$ on the complex $\Omega^\bullet_D(\mathcal{A})$ so that we get a chain complex $(\Omega^\bullet_D(\mathcal{A}),d)$ and a chain map 
$\pi_D:\Omega^\bullet ( \mathcal{A} ) \rightarrow \Omega^\bullet_D(\mathcal{A})$ such that the following diagram 
\begin{eqnarray}\label{induced differential of Connes}
\begin{tikzpicture}[node distance=3cm,auto]
\node (Up)[label=above:$\pi_D$]{};
\node (A)[node distance=1.5cm,left of=Up]{$\Omega^\bullet ( \mathcal{A} )$};
\node (B)[node distance=1.5cm,right of=Up]{$\Omega^\bullet_D(\mathcal{A})$};
\node (Down)[node distance=1.5cm,below of=Up, label=below:$\pi_D$]{};
\node(C)[node distance=1.5cm,left of=Down]{$\Omega^{\bullet+1} ( \mathcal{A} )$};
\node(D)[node distance=1.5cm,right of=Down]{$\Omega^{\bullet+1}_D(\mathcal{A})$};
\draw[->](A) to (B);
\draw[->](C) to (D);
\draw[->](B)to node{{ $d$}}(D);
\draw[->](A)to node[swap]{{ $d$}}(C);
\end{tikzpicture} 
\end{eqnarray}
commutes. This makes $\,\Omega^\bullet_D\,$ a differential graded algebra.

Let $(\mathcal{A}_1,\mathbb{V}_1,D_1,\gamma_1)$ and $(\mathcal{A}_2,\mathbb{V}_2,D_2,\gamma_2)$ be two even algebraic spectral triples. The product of these is given by the following even algebraic spectral triple
$$(\mathcal{A}_1,\mathbb{V}_1,D_1,\gamma_1)\otimes(\mathcal{A}_2,\mathbb{V}_2,D_2,\gamma_2)\,:=\,(\mathcal{A}_1\otimes \mathcal{A}_2,\mathbb{V}_1\otimes \mathbb{V}_2,D_1\otimes 1+\gamma_1\otimes D_2,\gamma_1\otimes
\gamma_2).$$ One can consider two dgas $\Omega_{D_1}^\bullet(\mathcal{A}_1)$ and $\Omega_{D_2}^\bullet(\mathcal{A}_2)$. The product of these two dgas is given by
$$\Omega_{D_1}^\bullet(\mathcal{A}_1)\otimes \Omega_{D_2}^\bullet(\mathcal{A}_2)\,:=\,\bigoplus_{n\geq 0}\bigoplus_{i+j=n} \Omega_{D_1}^i(\mathcal{A}_1)\otimes \Omega_{D_2}^j(\mathcal{A}_2)\, .$$
It is natural to ask how $\Omega_{D}^\bullet$ behaves under this multiplication, i,e. whether
$$\Omega_D^n(\mathcal{A}_1\otimes \mathcal{A}_2) \cong \bigoplus_{i+j=n} \Omega_{D_1}^i(\mathcal{A}_1)\otimes \Omega_{D_2}^j(\mathcal{A}_2)\, \, \, \forall\, n\geq 0\,.$$ 

Article (\cite{5}) deals with this investigation and does not lead to a final conclusion. However, using the universality of $\,\Omega^\bullet(\mathcal{A}_1\otimes \mathcal{A}_2)$, a useful outcome was that for all $\,n\geq 0$
\begin{eqnarray}\label{isomorphism of tensored complex}
 \Omega_D^n(\mathcal{A}_1\otimes \mathcal{A}_2) \cong \widetilde\Omega_D^n(\mathcal{A}_1,\mathcal{A}_2)\,,
\end{eqnarray}
where the description of $\,\widetilde\Omega_D^\bullet(\mathcal{A}_1,\mathcal{A}_2)$ is given by the following definition.

\begin{definition}\label{description of skew complex}
Consider the reduced universal dgas $\left(\Omega^\bullet(\mathcal{A}_1),d_1\right)$ and $\left(\Omega^\bullet(\mathcal{A}_2),d_2\right)$, associated with the algebraic spectral triples $(\mathcal{A}_1,
\mathbb{V}_1,D_1,\gamma_1)$ and $(\mathcal{A}_2,\mathbb{V}_2,D_2,\gamma_2)$ respectively. Consider the product dga $\left(\Omega^\bullet(\mathcal{A}_1)\otimes\Omega^\bullet(\mathcal{A}_2)\,,\widetilde d\,\right)$ where
\begin{eqnarray}\label{defn of new product}
 (\omega_i\otimes u_j).(\omega_p\otimes u_q):= (-1)^{jp}\omega_i\omega_p\otimes u_ju_q\,,
\end{eqnarray}
\begin{eqnarray}\label{defn of new d}
 \widetilde d(\omega_i\otimes u_j) := d_1(\omega_i)\otimes u_j+(-1)^i\omega_i\otimes d_2(u_j)\,,
\end{eqnarray}
for $\,\omega_\bullet\in\Omega^\bullet(\mathcal{A}_1)$ and $\,u_\bullet\in\Omega^\bullet(\mathcal{A}_2)$. One can define a representation $\,\widetilde \pi$ of $\,\Omega^\bullet(\mathcal{A}_1)\otimes\Omega^\bullet(\mathcal{A}_2)$ by
\begin{eqnarray}\label{defn of new pi}
 \widetilde \pi(\omega_i\otimes u_j) := \pi_1(\omega_i)\gamma_1^j\otimes \pi_2(u_j)\,.
\end{eqnarray}
Let 
\begin{eqnarray}\label{defn of new J}
 \widetilde J_0^k := Ker\left\{\widetilde \pi : \bigoplus_{i+j=k} \Omega^i(\mathcal{A}_1)\otimes \Omega^j(\mathcal{A}_2)\longrightarrow \mathcal{E}nd(\mathbb{V}_1\otimes \mathbb{V}_2)\right\}\,,
\end{eqnarray}
and $\widetilde J^n = \widetilde J_0^n + \widetilde d\widetilde J_0^{n-1}\,$. Define $\, \widetilde\Omega_D^n(\mathcal{A}_1,\mathcal{A}_2):= \frac{\bigoplus_{i+j=n} \Omega^i(\mathcal{A}_1)\otimes \Omega^j(\mathcal{A}_2)}
{\widetilde J^n (\mathcal{A}_1,\mathcal{A}_2)}\, ,\,\forall\,n\geq 0$.
\end{definition}

Note that $\,\widetilde\Omega_D^n(\mathcal{A}_1,\mathcal{A}_2)\cong \bigoplus_{i+j=n} \Omega_{D_1}^i(\mathcal{A}_1)\otimes \Omega_{D_2}^j(\mathcal{A}_2)$ if and only if
\begin{eqnarray}\label{our main investigation}
\widetilde J^n (\mathcal{A}_1,\mathcal{A}_2) & \cong & \bigoplus_{i+j=n} J^i(\mathcal{A}_1)\otimes \Omega^j(\mathcal{A}_2) + \Omega^i(\mathcal{A}_1)\otimes J^j(\mathcal{A}_2)\, . 
\end{eqnarray}
But it is in general not true. This is the prime investigation of this article. We propose a category theoretic construction of even algebraic spectral triples, which satisfies ($\, $\ref{our main investigation}$\, $).

\begin{definition}
The objects of the category $\widetilde{\mathcal{S}pec}$ are even algebraic spectral triples $(\mathcal{A},\mathbb{V},D,\gamma).$
Given two such objects $(\mathcal{A}_i,\mathbb{V}_i,D_i,\gamma_i),$ with $i=1,2,$ a morphism between them is a pair $(\phi,\Phi)$ where $\phi: \mathcal{A}_1\to \mathcal{A}_2$ is unital algebra morphism between the 
algebras $\mathcal{A}_1,\mathcal{A}_2$ and $\Phi \in \mathcal{E}nd(\mathbb{V}_1,\mathbb{V}_2)$ is surjective which intertwines the representations $\pi_1, \pi_2\circ \phi$ and the operators $D_1, D_2$ or equivalently the
following diagrams commute for every $x\in \mathcal{A}_1:$
\begin{center}
$\begin{array}{lcl}
\begin{tikzpicture}[node distance=3cm,auto]
\node (Up)[label=above:$\Phi$]{};
\node (A)[node distance=1.5cm,left of=Up]{$\mathbb{V}_1$};
\node (B)[node distance=1.5cm,right of=Up]{$\mathbb{V}_2$};
\node (Down)[node distance=1.5cm,below of=Up, label=below:$\Phi$]{};
\node(C)[node distance=1.5cm,left of=Down]{$\mathbb{V}_1$};
\node(D)[node distance=1.5cm,right of=Down]{$\mathbb{V}_2$};
\draw[->](A) to (B);
\draw[->](C) to (D);
\draw[->](B)to node{{ $D_2$}}(D);
\draw[->](A)to node[swap]{{ $D_1$}}(C);
\end{tikzpicture} 
\quad\quad\quad
\begin{tikzpicture}[node distance=3cm,auto]
\node (Up)[label=above:$\Phi$]{};
\node (A)[node distance=1.5cm,left of=Up]{$\mathbb{V}_1$};
\node (B)[node distance=1.5cm,right of=Up]{$\mathbb{V}_2$};
\node (Down)[node distance=1.5cm,below of=Up, label=below:$\Phi$]{};
\node(C)[node distance=1.5cm,left of=Down]{$\mathbb{V}_1$};
\node(D)[node distance=1.5cm,right of=Down]{$\mathbb{V}_2$};
\draw[->](A) to (B);
\draw[->](C) to (D);
\draw[->](B)to node{{ $\pi_2\circ \phi(x)$}}(D);
\draw[->](A)to node[swap]{{ $\pi_1(x)$}}(C);
\end{tikzpicture}
\end{array}$ 
\end{center}
and $\Phi$ also intertwines the grading operators $\gamma_1, \gamma_2$,
\begin{center}
\begin{tikzpicture}[node distance=3cm,auto]
\node (Up)[label=above:$\Phi$]{};
\node (A)[node distance=1.5cm,left of=Up]{$\mathbb{V}_1$};
\node (B)[node distance=1.5cm,right of=Up]{$\mathbb{V}_2$};
\node (Down)[node distance=1.5cm,below of=Up, label=below:$\Phi$]{};
\node(C)[node distance=1.5cm,left of=Down]{$\mathbb{V}_1$};
\node(D)[node distance=1.5cm,right of=Down]{$\mathbb{V}_2$};
\draw[->](A) to (B);
\draw[->](C) to (D);
\draw[->](B)to node{{ $\gamma_2$}}(D);
\draw[->](A)to node[swap]{{ $\gamma_1$}}(C);
\end{tikzpicture} 
\end{center}
\end{definition}

\begin{remark}
This definition is essentially from $($\cite{1}$)$. However, our requirement demands the extra condition on surjectivity of $\,\Phi\,$. This is in line with $($\cite{4.1},$\,$\cite{9}$)$.
\end{remark}

\begin{proposition}
The category $\,\widetilde{\mathcal{S}pec}\, $ is a monoidal category.
\end{proposition}
\begin{proof}
 Define the identity object `$1$' of monoidal category as follows
\begin{center}
 $1:=(\mathbb{K},\mathbb{K},0,1)$.
\end{center}
Define the functor tensor product `$\, \otimes\, $' on objects as 
\begin{eqnarray*}
&    & (\mathcal{A}_1,\mathbb{V}_1,D_1,\gamma_1)\otimes (\mathcal{A}_2,\mathbb{V}_2,D_2,\gamma_2)\\
& := & (\mathcal{A}_1\otimes \mathcal{A}_2,\mathbb{V}_1\otimes \mathbb{V}_2,D_1\otimes 1+ \gamma_1 \otimes D_2,\gamma_1\otimes \gamma_2)
\end{eqnarray*}
and on morphisms
\begin{center}
 $\,\,\,(\phi\,,\Phi):(\mathcal{A},\mathbb{V},D,\gamma)\longmapsto(\widetilde{\mathcal{A}},\widetilde{\mathbb{V}},\widetilde{D},\widetilde{\gamma})$
\end{center}
\begin{center} 
 $(\phi^\prime\,,\Phi^\prime):(\mathcal{A}^\prime,\mathbb{V}^\prime,D^\prime,\gamma^\prime)\longmapsto(\widetilde{\mathcal{A}}^\prime,\widetilde{\mathbb{V}}^\prime,\widetilde{D}^\prime,\widetilde{\gamma}^\prime)$
\end{center}
by $\, (\phi\otimes \phi^\prime,\Phi\otimes \Phi^\prime)\, $, where $\, \phi\otimes \phi^\prime\, $ is the usual tensor product of two algebra morphisms and $\, \Phi\otimes \Phi^\prime\, $ is the usual tensor product of two
linear maps. Now one can easily verify all the conditions of a monoidal category.
\end{proof}

Let $DGA$ be the category of differential graded algebras over field $\mathbb{K}\,$. We will only consider nonnegatively graded algebras in this article.

\begin{lemma}\label{Connes calculus is functor}
There is a covariant funtor $\,\mathcal{F}:\widetilde{\mathcal{S}pec}\longrightarrow DGA\,$ given by $\,(\mathcal{A},\mathbb{V},D,\gamma)\longmapsto\Omega_D^\bullet(\mathcal{A})$.
\end{lemma}
\begin{proof}
Consider two objects $(\mathcal{A}_1,\mathbb{V}_1,D_1,\gamma_1),(\mathcal{A}_2,\mathbb{V}_2,D_2,\gamma_2)\in \mathcal{O}b(\widetilde{\mathcal{S}pec})$ and suppose there is a morphism $(\phi\,,\Phi):(\mathcal{A}_1,\mathbb{V}_1,
D_1,\gamma_1)\longrightarrow (\mathcal{A}_2,\mathbb{V}_2,D_2,\gamma_2)$. Define $$\Psi:\Omega_{D_1}^\bullet(\mathcal{A}_1)\longrightarrow \Omega_{D_2}^\bullet(\mathcal{A}_2)$$ $$\quad\quad\quad\quad\,\left[\sum a_0\prod_{i=1}^
n[D_1,a_i]\right]\longmapsto\left[\sum \phi(a_0)\prod_{i=1}^n[D_2,\phi(a_i)]\right]$$ for all $\,a_j\in \mathcal{A}_j,\,n\geq 0\,$. To show $\Psi$ is well-defined we must show that $\Psi(\pi(d_1J_0^m))\subseteq \pi(d_2J_0^m)$
for all $m\geq 1$, where $d_1,d_2$ are the universal differentials for $\Omega^\bullet(\mathcal{A}_1),\Omega^\bullet(\mathcal{A}_2)$ respectively. Observe that
\begin{eqnarray}\label{intertwining relation}
 \Phi\circ \left(\sum a_0\prod_{i=1}^n[D_1,a_i]\right)=\left(\sum \phi(a_0)\prod_{i=1}^n[D_2,\phi(a_i)]\right)\circ \Phi\,.
\end{eqnarray}
Consider arbitrary element $\xi\in\pi(d_1J_0^n)$. By definition, $\,\xi=\sum \prod_{i=0}^n[D_1,a_i]\in \pi(d_1J_0^n)$ such that $\sum a_0\prod_{i=1}^n[D_1,a_i]=0$.
Now using equation ($\,$\ref{intertwining relation}$\,$) and surjectivity of $\Phi$, we have $$\sum \phi(a_0)\prod_{i=1}^n[D_2,\phi(a_i)]=0.$$ This shows well-definedness of $\Psi$. Now it is easy to check that $\Psi$ is a
differential graded algebra morphism. 
\end{proof}

\begin{remark}
This is the only place where we needed the stronger assumption on surjectivity of the map $\,\Phi$ and because of this reason we differ from $($\cite{1}$)$. 
\end{remark}

Now consider $(\mathcal{A},\mathbb{V},D,\gamma)\in\mathcal{O}b(\widetilde{\mathcal{S}pec})$ such that $\gamma\in\pi(\mathcal{A})$. Let $\,\widetilde{\mathcal{S}pec_{sub}}\,$ be the subcategory of $\,\widetilde{\mathcal{S}pec}
\,$, objects of which are $(\mathcal{A},\mathbb{V},D,\gamma)$ with $\gamma\in\pi(\mathcal{A})$. Clearly $\,\widetilde{\mathcal{S}pec_{sub}}\,$ is a monoidal subcategory of $\,\widetilde{\mathcal{S}pec}\,$. Now suppose
$(\mathcal{A},\mathbb{V},D,\gamma)\in\mathcal{O}b(\widetilde{\mathcal{S}pec})$ and $\gamma\notin\pi(\mathcal{A})$. Consider the vector space $\mathcal{A}\oplus\mathcal{A}$ with the product rule
$$(a,b)\star(\bar{a},\bar{b}):=(a\bar{a}+b\bar{b}\,,\,a\bar{b}+b\bar{a}).$$ The algebra $(\mathcal{A}\oplus\mathcal{A}\,,\,\star)$ becomes unital with unit $(1,0)$. The map $(a,b)\longmapsto(a+b,a-b)$ gives a unital algebra
isomorphism between the algebra $(\mathcal{A}\oplus\mathcal{A}\,,\,\star)$ and the direct sum algebra $\mathcal{A}\oplus\mathcal{A}$ where the multiplication is defined as co-ordinatewise. Now the map
$$(a,b)\longmapsto\pi(a)+\gamma\pi(b)\in\mathcal{E}nd(\mathbb{V})$$ gives a representation of the unital algebra $(\mathcal{A}\oplus\mathcal{A}\,,\,\star)$ on the vector space $\mathbb{V}$. Since $(0,1)\longmapsto\gamma
\in\mathcal{E}nd(\mathbb{V})$ we have $\gamma\in\pi((\mathcal{A}\oplus\mathcal{A}\,,\star\,))$ and hence $\left((\mathcal{A}\oplus\mathcal{A}\,,\,\star),\mathbb{V},D,\gamma\right)\in\mathcal{O}b(\widetilde{\mathcal{S}pec_{sub}})$.

\begin{proposition}\label{our construction}
The association $\,\mathcal{G}:(\mathcal{A},\mathbb{V},D,\gamma)\longmapsto((\mathcal{A}\oplus\mathcal{A}\,,\,\star),\mathbb{V},D,\gamma)$ gives a covariant functor from $\,\widetilde{\mathcal{S}pec}\,$ to
$\,\widetilde{\mathcal{S}pec_{sub}}\,$.
\end{proposition}
\begin{proof}
For a morphism $(\phi\,,\Phi):(\mathcal{A},\mathbb{V},D,\gamma)\longrightarrow(\mathcal{A}^\prime,\mathbb{V}^\prime,D^\prime,\gamma^\prime)$, define $$(\widetilde{\phi}\,,\widetilde{\Phi}):((\mathcal{A}\oplus\mathcal{A}\,,\star\,),
\mathbb{V},D,\gamma)\longrightarrow((\mathcal{A}^\prime\oplus\mathcal{A}^\prime\,,\star\,),\mathbb{V}^\prime,D^\prime,\gamma^\prime)$$ by taking $\,\widetilde{\Phi}:=\Phi\,$ and $$\widetilde{\phi}:\mathcal{A}\oplus\mathcal{A}
\longrightarrow\mathcal{A}^\prime\oplus\mathcal{A}^\prime$$ $$\quad\quad\quad\quad (a,b)\longmapsto (\phi(a),\phi(b)).$$ It is easy to check that $(\widetilde{\phi}\,,\widetilde{\Phi})$ defines a morphism in
$\widetilde{\mathcal{S}pec_{sub}}\,$.
\end{proof}

In the next two sections we will see that the funtor $\,\mathcal{F}\circ\mathcal{G}\,$ is not trivial. Throughout the rest of this article we will reserve the notation $\,\mathcal{F}\,$ and $\,\mathcal{G}\,$ to mean the
functors in Lemma $\,$\ref{Connes calculus is functor}$\,$ and Proposition $\,$\ref{our construction}$\,$ respectively.

\begin{theorem}\label{main theorem}
Restricted to the monoidal subcategory $\widetilde{\mathcal{S}pec_{sub}}$ of $\,\widetilde{\mathcal{S}pec}\,$, the covariant functor $\,\mathcal{F}:\widetilde{\mathcal{S}pec_{sub}}\longrightarrow DGA\,$ defined in Lemma
$\,\ref{Connes calculus is functor}\,$ is a monoidal functor.
\end{theorem}
\begin{proof}
Only nontrivial part is to prove that $$\Omega_D^n(\mathcal{A}_1\otimes\mathcal{A}_2) \cong \bigoplus_{i+j=n}\Omega_{D_1}^i(\mathcal{A}_1)\otimes \Omega_{D_2}^j(\mathcal{A}_2)\,.$$ where
$\,D=D_1\otimes 1+\gamma_1\otimes D_2\, $. We break the proof into two lemmas.

\begin{lemma}\label{[D^2,. ] (1)}
For any $a\in\mathcal{A}\,,\,[D^2,a]\in\pi(dJ_0^1)\,$.
\end{lemma}
\begin{proof}
Consider $p=(1+\gamma)/2$ and $q=(1-\gamma)/2$. Then $pq=0$ and $pDp=qDq=0$. Consider $a\in\mathcal{A}$ and $\eta\in\mathcal{A}$ be such that $\pi(\eta)=\gamma$. Now consider $\,\omega=\frac{1}{4}(1+\eta)d(a)(1+\eta)+
\frac{1}{4}(1-\eta)d(a)(1-\eta)$ in $\,\Omega^1(\mathcal{A})$. Then,
\begin{eqnarray*}
 \pi(\omega) & = & p[D,ap] - pa[D,p] + q[D,aq] - qa[D,q]\\
             & = & pDap - papD - paDp + papD + qDaq - qaqD - qaDq + qaqD\\
             & = & 0\, ;\, \, \, \, \, since \, \, \, pap=pa=ap\, ;\, qaq=qa=aq.
\end{eqnarray*}
This shows that $\,\omega\in J_0^1(\mathcal{A})\subseteq\Omega^1(\mathcal{A})$. Now,
\begin{eqnarray*}
 \pi(d\omega) & = & [D,p][D,ap] - [D,pa][D,p] + [D,q][D,aq] - [D,qa][D,q]\\
              & = & (Dp-pD)(Dap-apD) - Dpa[D,p] + paD[D,p]\\
              &   & + (Dq-qD)(Daq-aqD) - Dqa[D,q] + qaD[D,q]\\
              & = & DpDap - DpapD - pD^2ap + pDapD - DpaDp + DpapD\\
              &   & + paD^2p - paDpD + DqDaq - DqaqD - qD^2aq\\
              &   & + qDaqD - DqaDq + DqaqD + qaD^2q - qaDqD\\
              & = & 0 - DpaD - pD^2ap + 0 - 0 + DpaD + paD^2p - 0\\
              &   & + 0 - DqaD - qD^2aq + 0 - 0 + DqaD + qaD^2q - 0\\
              & = & - [pD^2p,pa] - [qD^2q,qa]
\end{eqnarray*}
Now observe that $pD^2q = pD(p+q)Dq = 0$ and $qD^2p = 0$. Hence,
\begin{eqnarray*}
 [D^2,a] & = & [(p+q)D^2(p+q),(p+q)a]\\
           & = & [pD^2p+qD^2q,pa+qa]\\
           & = & [pD^2p,pa] + [qD^2q,qa]\\
           & = & \pi(d\omega)
\end{eqnarray*}
This proves that $\,[D^2,a]\in\pi(dJ_0^1(\mathcal{A}))$.
\end{proof}

\begin{lemma}\label{[D^2,. ] (2)}
We have $$\widetilde J^n(\mathcal{A}_1,\mathcal{A}_2)\,=\,\bigoplus_{i+j=n}J^i(\mathcal{A}_1)\otimes\Omega^j(\mathcal{A}_2)+\Omega^i(\mathcal{A}_1)\otimes J^j(\mathcal{A}_2)\,,$$ where definition of $\,\widetilde J^n$ is
provided in definition $\,\ref{description of skew complex}\,$.
\end{lemma}
\begin{proof}
Let $\omega=\widetilde \pi (\widetilde d\omega^\prime)$ where $\omega^\prime \in \widetilde J_0^{n-1}$. Suppose $\,\omega^\prime = \sum_k\bigoplus_{i+j=n-1}v_{1,k}^i\otimes v_{2,k}^j\,$, where $\,v_{1,k}^i\in\Omega^i
(\mathcal{A}_1)$ and $\,v_{2,k}^j\in\Omega^j(\mathcal{A}_2)$. Hence we have the following equation,
\begin{eqnarray}\label{eqn 1 in [D^2,. ] (2)}
 \sum_k \sum_{i+j=n-1} \pi_1(v_{1,k}^i)\gamma_1^j\otimes \pi_2(v_{2,k}^j) = 0\, .
\end{eqnarray}
Let
\begin{center}
 $v_{1,k}^i = \sum a_{0,k}^{(i)}\prod_{r=1}^id_1(a_{r,k}^{(i)})$
\end{center}
\begin{center}
 $v_{2,k}^j = \sum b_{0,k}^{(j)}\prod_{s=1}^jd_2(b_{s,k}^{(j)})$
\end{center}
for $\,a_{r,k}^{(i)}\in\mathcal{A}_1$ and $\,b_{s,k}^{(j)}\in\mathcal{A}_2$. Then equation ($\, $\ref{eqn 1 in [D^2,. ] (2)}$\, $) becomes
\begin{eqnarray}\label{eqn 2 in [D^2,. ] (2)}
 \sum_k \sum_{i+j=n-1} \sum \left(a_{0,k}^{(i)}\prod_{r=1}^i[D_1,a_{r,k}^{(i)}]\gamma_1^j\right)\otimes \left(b_{0,k}^{(j)}\prod_{s=1}^j[D_2,b_{s,k}^{(j)}]\right) = 0\, .
\end{eqnarray}
Now since $\omega^\prime = \sum_k \bigoplus_{i+j=n-1}v_{1,k}^i\otimes v_{2,k}^j\,$,
\begin{center}
 $\widetilde d(\omega^\prime) = \sum_k \sum_{i+j=n-1} d_1(v_{1,k}^i)\otimes v_{2,k}^j + (-1)^i v_{1,k}^i\otimes d_2(v_{2,k}^j)$
\end{center}
and hence,
\begin{eqnarray*}
\widetilde \pi (\widetilde d\omega^\prime) & = & \sum_k \sum_{i+j=n-1} \pi_1(d_1(v_{1,k}^i))\gamma^j\otimes \pi_2(v_{2,k}^j) + (-1)^i \pi_1(v_{1,k}^i)\gamma^{j+1}\otimes \pi_2(d_2(v_{2,k}^j)).
\end{eqnarray*}
Using equation ($\, $\ref{eqn 2 in [D^2,. ] (2)}$\, $) we get,
\begin{eqnarray*}
 &   & \sum_k \sum_{i+j=n-1} \pi_1(d_1(v_{1,k}^i))\gamma^j\otimes \pi_2(v_{2,k}^j)\\
 & = & \sum_k \sum_{i+j=n-1} \sum \left(-a_{0,k}^{(i)}D_1\prod_{r=1}^i[D_1,a_{r,k}^{(i)}]\gamma_1^j\right)\otimes \left(b_{0,k}^{(j)}\prod_{s=1}^j[D_2,b_{s,k}^{(j)}]\right)\\
 & = & - \sum_k \sum_{i+j=n-1} \sum \textbf{\{}\sum_{r=1}^i \left((-1)^{r+1} a_{0,k}^{(i)}[D_1,a_{1,k}^{(i)}]\ldots [D_1^2,a_{r,k}^{(i)}]\ldots [D_1,a_{i,k}^{(i)}]\gamma_1^j\right)\\
 &   & \quad\quad\quad\otimes \left(b_{0,k}^{(j)}\prod_{s=1}^j[D_2,b_{s,k}^{(j)}]\right) - \left((-1)^i a_{0,k}^{(i)}\prod_{t=1}^i[D_1,a_{t,k}^{(i)}]D_1 \gamma_1^j\right)\otimes \left(b_{0,k}^{(j)}\prod_{s=1}^j[D_2,b_{s,k}^{(j)}]\right)\textbf{\}}\,.
 \end{eqnarray*}
This term is contained in $\,\sum_{i+j=n}\pi_1(J^i)\gamma_1^j\otimes\pi_2(\Omega^j)\,$, since we have seen that $[D_1^2,a_{r,k}^{(i)}]$ is in $J^2$ for each $1\leqslant r\leqslant i\,\,(\,$Lemma~~\ref{[D^2,. ] (1)}$\,$). Finally,
\begin{eqnarray*}
&   & \sum_k \sum_{i+j=n-1} (-1)^i \pi_1(v_{1,k}^i)\gamma^{j+1}\otimes \pi_2(d_2(v_{2,k}^j))\\
& = & \sum_k \sum_{i+j=n-1} \sum \left[\gamma_1\otimes D_2\, ,\, \left(a_{0,k}^{(i)}\prod_{t=1}^i[D_1,a_{t,k}^{(i)}]\gamma_1^j\right)\otimes \left(b_{0,k}^{(j)}\prod_{s=1}^j[D_2,b_{s,k}^{(j)}]\right)\right]\\
&   & \quad\quad\quad+ \sum_{r=1}^j (-1)^{r+1} \left(a_{0,k}^{(i)}\prod_{t=1}^i[D_1,a_{t,k}^{(i)}]\gamma_1^{j+1}\right)\otimes \left(b_{0,k}^{(j)}[D_2,b_{1,k}^{(j)}]\ldots [D_2^2,b_{r,k}^{(j)}]\ldots [D_2,b_{j,k}^{(j)}]\right)\,.
\end{eqnarray*}
This term is in $\,\sum_{i+j=n} \pi_1(\Omega^i)\gamma_1^j\otimes \pi_2(J^j)\,$, since $[D_2^2,b_{r,k}^{(j)}] \in J^2$ for each $1\leqslant r \leqslant j\, \, (\, $\ref{[D^2,. ] (1)}$\, $). So we get
\begin{center}
 $\widetilde \pi(\widetilde{J}^n)\, \subseteq\, \sum_{i+j=n} \pi_1(\Omega^i)\gamma_1^j\otimes \pi_2(J^j) + \pi_1(J^i)\gamma_1^j\otimes \pi_2(\Omega^j)\,$,
\end{center}
i,e. $\, \widetilde \pi(\widetilde{J}^n)\, \subseteq\, \widetilde \pi \left(\bigoplus_{i+j=n} (J^i\otimes \Omega^j + \Omega^i\otimes J^j)\right)$.

Hence,
\begin{center}
$\frac{\widetilde{J}^n}{\bigoplus_{i+j=n}(J^i\otimes\Omega^j+\Omega^i\otimes J^j)}\cong\frac{\widetilde\pi(\widetilde{J}^n)}{\widetilde\pi\left(\bigoplus_{i+j=n}(J^i\otimes\Omega^j+\Omega^i\otimes J^j)\right)} = \{0\}$
\end{center}
and our claim has been justified.
\end{proof}

\textbf{ Proof of Theorem \ref{main theorem}~:} Lemma $\, $\ref{[D^2,. ] (2)}$\,$ proves that the isomorphism in equation ($\,$\ref{our main investigation}$\,$) holds i,e.
$$\widetilde{J}^n(\mathcal{A}_1\,,\,\mathcal{A}_2) \cong \bigoplus_{i+j=n}J^i(\mathcal{A}_1)\otimes\Omega^j(\mathcal{A}_2)+\Omega^i(\mathcal{A}_1)\otimes J^j(\mathcal{A}_2)\,,$$
when we restrict ourselves to the subcategory $\widetilde{\mathcal{S}pec_{sub}}$. Hence the proof follows from the fact that $\Omega_D^n(\mathcal{B}_1\otimes\mathcal{B}_2)\cong\widetilde{\Omega}_D^n(\mathcal{B}_1,
\mathcal{B}_2)$ for all $n\geq 0$ and for any unital algebras $\mathcal{B}_1, \mathcal{B}_2$ (see the isomorphism in $\,$\ref{isomorphism of tensored complex}$\, $). 
\end{proof}

\begin{corollary}
$\,\mathcal{F}(\mathcal{G}(\mathcal{A}_1)\otimes\mathcal{G}(\mathcal{A}_2))\cong\mathcal{F}\circ\mathcal{G}(\mathcal{A}_1)\otimes\mathcal{F}\circ\mathcal{G}(\mathcal{A}_2)$.
\end{corollary}

However, we do not know whether $\,\mathcal{F}\circ\mathcal{G}(\mathcal{A}_1\otimes\mathcal{A}_2)\cong\mathcal{F}\circ\mathcal{G}(\mathcal{A}_1)\otimes\mathcal{F}\circ\mathcal{G}(\mathcal{A}_2)$.
\medskip


\section{Computation for Compact Manifold}
In this section we show that there exists a contravariant functor $\,\mathcal{P}\,$ from the category of manifolds with embeddings as morphisms to the category $\,\widetilde{\mathcal{S}pec}$ and we show that $\,\mathcal{F}\circ
\mathcal{G}\circ\mathcal{P}\,$ is not trivial.

Let $\mathbb{M}$ be a compact manifold of dimension $n$ with atlas $\{U_i,\phi_i\}_{i=1}^k$. Consider the complexified exterior bundle $\wedge^\bullet_\mathbb{C} T^*\mathbb{M}$ over $\mathbb{M}$ and $(x^1,\ldots,x^n)\,$
denotes the local co-ordinates. Let $d$ be the exterior differentiation. If we consider the category of manifolds $\mathcal{M}$ with embeddings as morphisms (\cite{4.1}, \cite{9}), then there is a contravariant functor from
$\mathcal{M}$ to $\,\widetilde{\mathcal{S}pec}$. To see this consider the following object $$\left(C^\infty(\mathbb{M})\,,\,\Gamma(\wedge^\bullet_\mathbb{C} T^*\mathbb{M})\cong\Gamma(\wedge^{even}_\mathbb{C}T^*\mathbb{M})
\oplus\Gamma(\wedge^{odd}_\mathbb{C} T^*\mathbb{M})\,,\,D:=\begin{pmatrix}
                                                                   0 & d\\
                                                                   d & 0
                                                                  \end{pmatrix}\,,\,\gamma:=parity\right)$$ in $\widetilde{\mathcal{S}pec}\,$, where `parity' means the odd-even parity of a form in $\Gamma(\wedge^\bullet_\mathbb{C} T^*
\mathbb{M})$. Now for an embedding $\,\phi:\mathbb{M}\hookrightarrow\mathbb{N}$, we have $$\widetilde{\phi}:\left(C^\infty(\mathbb{N})\,,\,\Gamma(\wedge^\bullet_\mathbb{C} T^*\mathbb{N})\,,\,\begin{pmatrix}
                                                                   0 & d\\
                                                                   d & 0
                                                                  \end{pmatrix}\,,\,\gamma\right)\longrightarrow\left(C^\infty(\mathbb{M})\,,\,\Gamma(\wedge^\bullet_\mathbb{C} T^*\mathbb{M})\,,\,\begin{pmatrix}
                                                                   0 & d\\
                                                                   d & 0
                                                                  \end{pmatrix}\,,\,\gamma\right)\,,$$ a morphism in $\,\widetilde{\mathcal{S}pec}$. Moreover, the following commutative diagram
\begin{center}
\begin{tikzpicture}[node distance=3cm,auto]
\node (Up)[label=above:$d$]{};
\node (A)[node distance=2cm,left of=Up]{$\Gamma(\wedge^k_\mathbb{C} T^*\mathbb{N})$};
\node (B)[node distance=2cm,right of=Up]{$\Gamma(\wedge^{k+1}_\mathbb{C} T^*\mathbb{N})$};
\node (Down)[node distance=1.5cm,below of=Up, label=below:$d$]{};
\node(C)[node distance=2cm,left of=Down]{$\Gamma(\wedge^k_\mathbb{C} T^*\mathbb{M})$};
\node(D)[node distance=2cm,right of=Down]{$\Gamma(\wedge^{k+1}_\mathbb{C} T^*\mathbb{M})$};
\draw[->](A) to (B);
\draw[->](C) to (D);
\draw[->](B)to node{{ $\phi^*$}}(D);
\draw[->](A)to node[swap]{{ $\phi^*$}}(C);
\end{tikzpicture}
\end{center} where $\,\phi^*$ is the pullback of $\,\phi\,$, explains that the consideration of the quadruple $\left(C^\infty(\mathbb{M}),\Gamma(\wedge^\bullet_\mathbb{C} T^*\mathbb{M}),D,deg\right)$ is natural. Henceforth
we will be dealing with $\left(C^\infty(\mathbb{M}),\Gamma(\wedge^\bullet_\mathbb{C} T^*\mathbb{M}),D,\gamma\right)\in\mathcal{O}b\left(\widetilde{\mathcal{S}pec}\right)$ in this section, where $D=\begin{pmatrix}
 0 & d\\
 d & 0
\end{pmatrix}$ and $\gamma=\begin{pmatrix}
1 & 0\\
0 & -1
\end{pmatrix}$. Notice that $D^2=0$.  Since $\gamma\notin\pi(C^\infty(\mathbb{M}))$ we first apply the functor $\,\mathcal{G}\,$ of Proposition $\,$\ref{our construction}$\,$ and then compute $\,\mathcal{F}\circ\mathcal{G}\,$
along with the associated cohomologies.
\medskip

\textbf{Notation~:} $\widetilde{C^\infty(\mathbb{M})}:= \mathcal{G}\left(C^\infty(\mathbb{M}),\Gamma(\wedge^\bullet_\mathbb{C} T^*\mathbb{M}),D,\gamma\right)$ where $\,\mathcal{G}\,$ is defined in Proposition
$\,$\ref{our construction}$\,$ and $\,dim(\mathbb{M})=n\,$ throughout this section.
\medskip

\begin{lemma}\label{denominator vanishing}
$\Omega_D^m\left(\widetilde{C^\infty(\mathbb{M})}\right)\cong \pi\left(\Omega^m(\widetilde{C^\infty(\mathbb{M})})\right)\, \, \forall\, m\geq 0$.
\end{lemma}
\begin{proof}
Observe that $J_0^0\left(\widetilde{C^\infty(\mathbb{M})}\right)=\{0\}$ in this case. We show that $\pi\left(dJ_0^m(\widetilde{C^\infty(\mathbb{M})})\right)=\{0\}\, \, \forall\, m\geq 1$. Note that
\begin{eqnarray*}
 \pi(dJ_0^m) & = & \{\, \sum \prod_{i=0}^m [D,x_i]:x_i\in \widetilde{C^\infty(\mathbb{M})}\,\,;\,\,\sum x_0\prod_{i=1}^m [D,x_i]=0\}\\
             & = & \{\, -\, \sum x_0D\prod_{i=1}^m [D,x_i]:x_i\in \widetilde{C^\infty(\mathbb{M})}\,\,;\,\,\sum x_0\prod_{i=1}^m [D,x_i]=0\}
\end{eqnarray*}
Now,
\begin{eqnarray*}
 \sum x_0D\prod_{i=1}^m [D,x_i] & = & \sum x_0D\prod_{i=1}^m(Dx_i-x_iD)\\
                                & = & -\sum x_0Dx_1D\prod_{i=2}^m(Dx_i-x_iD)\\
                                & = & (-1)^m\sum x_0\prod_{i=1}^mDx_iD\\
                                & = & (-1)^m\left(\sum x_0\prod_{i=1}^{m-1}[D,x_i]\right)Dx_mD\\
                                & = & (-1)^m\left(\sum x_0\prod_{i=1}^m[D,x_i]\right)D
\end{eqnarray*}
But $\,\sum x_0\prod_{i=1}^m [D,x_i]=0\,$ by assumption and hence we are done.
\end{proof}

Let $1\leq m\leq n$, where $n=dim(\mathbb{M})$. We define the following linear operator $$T_{a_0,\ldots,a_m}\,:\,\Gamma(\wedge^\bullet_\mathbb{C} T^*\mathbb{M})\longrightarrow
\Gamma(\wedge^\bullet_\mathbb{C} T^*\mathbb{M})$$ $$\quad\quad\quad\quad\quad\quad\omega\longmapsto\,a_0da_1\wedge\ldots \wedge d(a_m\omega)$$ where $\,a_i\in C^\infty(\mathbb{M})$. Let $\mathcal{M}_m=span\{T_{a_0,\ldots,a_m}
:\Gamma(\wedge^\bullet_\mathbb{C} T^*\mathbb{M})\longrightarrow \Gamma(\wedge^\bullet_\mathbb{C}T^*\mathbb{M})\,\colon a_i\in C^\infty(\mathbb{M})\}$. Then $\mathcal{M}_m$ is a $\mathbb{C}\,$-vector space. Note that for
$a,b\in C^\infty(\mathbb{M})$ $$\left[D,\begin{pmatrix}
                                         a & 0\\
                                         0 & b
                                        \end{pmatrix}\right]=\begin{pmatrix}
                                         0 & T_{1,b}-T_{a,1}\\
                                         T_{1,a}-T_{b,1} & 0
                                        \end{pmatrix}.$$ Since elements of
$\pi\left(\Omega^m(\widetilde{C^\infty(\mathbb{M})})\right)$ are of the form $$\sum \begin{pmatrix}
                                         a_0 & 0\\
                                         0 & b_0
                                        \end{pmatrix}\prod_{i=1}^m\left[D,\begin{pmatrix}
                                         a_i & 0\\
                                         0 & b_i
                                        \end{pmatrix}\right]\,;\,a_j,b_j\in C^\infty(\mathbb{M})\,,$$ it is easy to observe that $\pi\left(\Omega^m(\widetilde{C^\infty(\mathbb{M})})\right)$ is a subspace of
$\mathcal{M}_m\oplus\mathcal{M}_m$. Moreover, using the equality
\begin{eqnarray}\label{vector space of 1-form}
\sum_k\begin{pmatrix}
  a_{0k} & 0\\
  0 & a_{0k}^\prime
\end{pmatrix}\left[D,\begin{pmatrix}
  0 & 0\\
  0 & 1
\end{pmatrix}\right]\begin{pmatrix}
  -a_{1k}^\prime & 0\\
  0 & a_{1k}
\end{pmatrix} & = & \sum_k\begin{pmatrix}
  0 & T_{a_{0k}\,,\,a_{1k}}\\
  T_{a_{0k}^\prime\,,\,a_{1k}^\prime} & 0
\end{pmatrix}\,.
\end{eqnarray}
 we see that for $m\geq 3$ odd
\begin{eqnarray*}
&   & \sum\begin{pmatrix}
  0 & T_{a_0,a_1,\ldots ,a_m}\\
  T_{a_0^\prime ,a_1^\prime,\ldots ,a_m^\prime} & 0
\end{pmatrix}\\
& = & \sum\begin{pmatrix}
  0 & T_{a_0,a_1}\\
  T_{a_0^\prime ,a_1^\prime} & 0
\end{pmatrix}\prod_{i=2,i\, \, even}^m \left(\begin{pmatrix}
  0 & T_{1,a_i^\prime}\\
  T_{1,a_i} & 0
\end{pmatrix}\begin{pmatrix}
  0 & T_{1,a_{i+1}}\\
  T_{1,a_{i+1}^\prime} & 0
\end{pmatrix}\right)\\
& = & \sum\begin{pmatrix}
  a_0 & 0\\
  0 & a_0^\prime
\end{pmatrix}\left[D,\begin{pmatrix}
  0 & 0\\
  0 & 1
\end{pmatrix}\right]\begin{pmatrix}
  -a_1^\prime & 0\\
  0 & a_1
\end{pmatrix}\bullet\\
&   & \quad\prod_{i=2,i\, \, even}^m \left(\left\{\left[D,\begin{pmatrix}
  0 & 0\\
  0 & 1
\end{pmatrix}\right]\begin{pmatrix}
  -a_i & 0\\
  0 & a_i^\prime
\end{pmatrix}\right\} \left\{\left[D,\begin{pmatrix}
  0 & 0\\
  0 & 1
\end{pmatrix}\right]\begin{pmatrix}
  -a_{i+1}^\prime & 0\\
  0 & a_{i+1}
\end{pmatrix}\right\}\right)
\end{eqnarray*}
and similarly for $m\geq 2$ even. Hence we conclude that $\pi\left(\Omega^m\left(\widetilde{C^\infty(\mathbb{M})}\right)\right)=\mathcal{M}_m\oplus\mathcal{M}_m$.

\begin{lemma}\label{to refer in next lemma}
Let $\mathbb{V}$ be the vector space of all linear endomorphisms acting on $\Gamma(\wedge^\bullet_\mathbb{C} T^*\mathbb{M})$. We have the following subspaces of $\,\mathbb{V}$ $$\quad\quad\quad\,\,\quad\mathcal{M}_m^{(1)}:=\{
M_{\omega_{m-1}}\circ\,d:\Gamma(\wedge^\bullet_\mathbb{C} T^*\mathbb{M})\longrightarrow \Gamma(\wedge^\bullet_\mathbb{C} T^*\mathbb{M})\,\colon\,\omega_{m-1}\in\Gamma(\wedge^{m-1}_\mathbb{C} T^*\mathbb{M})\}\,,$$
$$\mathcal{M}_m^{(2)}:=\{M_{\omega_m}:\Gamma(\wedge^\bullet_\mathbb{C} T^*\mathbb{M})\longrightarrow \Gamma(\wedge^\bullet_\mathbb{C} T^*\mathbb{M})\,\colon\,\omega_m\in\Gamma(\wedge^m_\mathbb{C} T^*\mathbb{M})\}\,.$$
where $M_\xi$ denotes multiplication by $\xi\,$. Then $\mathcal{M}_m^{(1)}\bigcap\mathcal{M}_m^{(2)}=\{0\}$ and $\mathcal{M}_m\subseteq\mathcal{M}_m^{(1)}\oplus \mathcal{M}_m^{(2)}\subseteq\mathbb{V}$.
\end{lemma}
\begin{proof}
Observe that $T_{a_0,\ldots ,a_m}(\omega)=(M_{a_0a_mda_1\wedge\ldots\wedge da_{m-1}}\circ d+M_{a_0da_1\wedge\ldots\wedge da_m})(\omega),\,\forall\omega\in\Gamma(\wedge^\bullet_\mathbb{C} T^*\mathbb{M})$. Since $d(1)=0$
and $\wedge(1)=1$, we have the direct sum.
\end{proof}

\begin{lemma}\label{bijection for m}
For $1\leq m\leq n$ define $$\widetilde \Phi:\mathcal{M}_m\,\,\longrightarrow \Omega^{m-1}\mathbb{M}\oplus \Omega^m\mathbb{M}$$
$$\quad\quad\quad\quad\quad\quad\quad\quad\quad\quad\quad\quad T_{a_0,\ldots,a_m}\longmapsto (a_0a_mda_1\wedge \ldots \wedge da_{m-1}\,,\,a_0da_1\wedge \ldots \wedge da_m)$$
where $\,\Omega^k\mathbb{M}:=\Gamma(\wedge^k_\mathbb{C} T^*\mathbb{M})$ denotes the space of $k$-forms on $\mathbb{M}$. Then $$\quad\Phi=(\widetilde \Phi,\widetilde \Phi):\mathcal{M}_m\oplus \mathcal{M}_m\longrightarrow
\Omega^{m-1}\mathbb{M}\oplus \Omega^m\mathbb{M}\oplus\Omega^{m-1}\mathbb{M}\oplus \Omega^m\mathbb{M}$$ is a linear bijection.  
\end{lemma}
\begin{proof}
Observe that to prove well-definedness of $\,\widetilde{\Phi}\,$, in view of Lemma \ref{to refer in next lemma}~, we only need to show that for $0\leq k\leq n-1$, if $M_{\omega_k}\circ d$ is zero then $\omega_k=0$. In a
co-ordinate neighbourhood around a point $p\in\mathbb{M}$, suppose $\omega_k=\sum_{i_1<\ldots<i_k} f_{i_1,\ldots,i_k}dx^{i_1}\wedge\ldots\wedge dx^{i_k}$. Since $k\leq n-1$, there always exist $j\notin\{i_1,\ldots,i_k\}$ and
we have $\omega_k\wedge dx^j=0$ i,e. $\sum_{i_1<\ldots<i_k} f_{i_1,\ldots,i_k}dx^{i_1}\wedge\ldots\wedge dx^{i_k}\wedge dx^j=0$ at each point of the co-ordinate neighbourhood around $p\in\mathbb{M}$. This will show that each
$f_{i_1,\ldots,i_k}$ is zero showing $\,\omega_k=0$. Injectivity of $\widetilde \Phi$ easily follows from Lemma \ref{to refer in next lemma}~. To see surjectivity, choose $(\omega_{m-1},\omega_m)\in \Omega^{m-1}\mathbb{M}
\oplus \Omega^m\mathbb{M}$. Let in a co-ordinate neighbourhood
\begin{center}
 $\omega_m=\,\sum_{i_1<\ldots <i_m} f_{i_1\ldots i_m}dx^{i_1}\wedge \ldots \wedge dx^{i_m}$
\end{center}
\begin{center}
 $\quad\quad\omega_{m-1}=\,\sum_{j_1<\ldots <j_{m-1}} g_{j_1\ldots j_{m-1}}dx^{j_1}\wedge \ldots \wedge dx^{j_{m-1}}$
\end{center}
with support of $\,f_{i_1\ldots i_m}\,,\,g_{j_1\ldots j_{m-1}}$ in that neighbourhood. Then
\begin{center}
 $\, \sum T_{g_{j_1\ldots j_{m-1}},x^{j_1},\ldots ,x^{j_{m-1}},1}\longmapsto (\omega_{m-1},0)$
\end{center}
\begin{center}
 $\, \sum T_{f_{i_1\ldots i_m},x^{i_1},\ldots ,x^{i_m}}\longmapsto (\sum f_{i_1\ldots i_m}x^{i_m}dx^{i_1}\wedge \ldots \wedge dx^{i_{m-1}},\omega_m)$
\end{center}
This shows that
\begin{eqnarray*}
\widetilde \Phi^{-1} (\omega_{m-1},\omega_m) & = & \sum T_{f_{i_1\ldots i_m},x^{i_1},\ldots ,x^{i_m}}+\sum T_{g_{j_1\ldots j_{m-1}},x^{j_1},\ldots ,x^{j_{m-1}},1}\\
                                             &   & \quad- \sum T_{f_{i_1\ldots i_m}x^{i_m},x^{i_1},\ldots ,x^{i_{m-1}},1}
\end{eqnarray*}
and containment of support of the functions $f_{i_1\ldots i_m}$ and $g_{j_1\ldots j_{m-1}}$ in the co-ordinate neighbourhood fulfills our claim.
\end{proof}

\begin{lemma}\label{higher forms vanish}
For all $m\geq n+1$, where $n=dim(\mathbb{M})$, $\,\mathcal{M}_m=\{0\}$.
\end{lemma}
\begin{proof}
Note that for any $\,\omega\in\Gamma(\wedge^\bullet_\mathbb{C} T^*\mathbb{M})$, $$T_{a_0,\ldots,a_m}(\omega):=\left(M_{a_0a_mda_1\wedge\ldots \wedge da_{m-1}}\circ d+M_{a_0da_1\wedge\ldots \wedge da_m}\right)
(\omega).$$ Since $m\geq n+1$, it follows that $\mathcal{M}_m=\{0\}$ because $\Omega^k\mathbb{M}=\Gamma(\wedge^k_\mathbb{C} T^*\mathbb{M})=\{0\}$ for all $k>n$. 
\end{proof}

\begin{proposition}\label{bimodule action}
$\Omega^{m-1}\mathbb{M}\, \oplus\, \Omega^m\mathbb{M}\, \oplus\, \Omega^{m-1}\mathbb{M}\, \oplus\, \Omega^m\mathbb{M}\, $ has a $\widetilde{C^\infty(\mathbb{M})}$-bimodule structure given by,
\begin{eqnarray*}
 \begin{pmatrix}
 \phi & 0\\
 0 & \psi
\end{pmatrix}\, .\, (\omega_{m-1},\omega_m,\widetilde \omega_{m-1},\widetilde \omega_m) & := & (\phi \omega_{m-1},\phi \omega_m,\psi \widetilde \omega_{m-1},\psi \widetilde \omega_m)
\end{eqnarray*}
\begin{eqnarray*}
&   & (\omega_{m-1},\omega_m,\widetilde \omega_{m-1},\widetilde \omega_m)\, .\,  \begin{pmatrix}
 \phi & 0\\
 0 & \psi
\end{pmatrix}\\
& := & \begin{cases}
        \begin{array}{lcl}
          (\phi \omega_{m-1},\phi \omega_m-d\phi \wedge \omega_{m-1},\psi \widetilde \omega_{m-1},\psi \widetilde \omega_m-d\psi \wedge \widetilde \omega_{m-1})\, \, ;\,\,\,if\, \, m\, \, even\\
          (\psi \omega_{m-1},\psi \omega_m+d\psi \wedge \omega_{m-1},\phi \widetilde \omega_{m-1},\phi \widetilde \omega_m+d\phi \wedge \widetilde \omega_{m-1})\, \, ;\,\,\,if\, \, m\, \, odd
         \end{array}
        \end{cases}
\end{eqnarray*}
\end{proposition}
\begin{proof}
In co-ordinate chart, $$\quad\quad\omega_{m-1}=\,\sum_{j_1<\ldots <j_{m-1}}g_{j_1\ldots j_{m-1}}dx^{j_1}\wedge\ldots\wedge dx^{j_{m-1}}$$ $$\omega_m=\,\sum_{i_1<\ldots <i_m}f_{i_1\ldots i_m}dx^{i_1}\wedge\ldots\wedge dx^{i_m}$$
$$\quad\quad\widetilde\omega_{m-1}=\,\sum_{j_1<\ldots <j_{m-1}}\widetilde{g_{j_1\ldots j_{m-1}}}dx^{j_1}\wedge\ldots\wedge dx^{j_{m-1}}$$ $$\widetilde\omega_m=\,\sum_{i_1<\ldots <i_m} \widetilde{f_{i_1\ldots i_m}}dx^{i_1}\wedge
\ldots\wedge dx^{i_m}$$ Alos let
\begin{eqnarray*}
 \xi & = & \sum T_{f_{i_1\ldots i_m},x^{i_1},\ldots ,x^{i_m}}+\sum T_{g_{j_1\ldots j_{m-1}},x^{j_1},\ldots ,x^{j_{m-1}},1}\\
     &   & \quad - \sum T_{f_{i_1\ldots i_m}x^{i_m},x^{i_1},\ldots ,x^{i_{m-1}},1}
\end{eqnarray*}
and
\begin{eqnarray*}
 \widetilde \xi & = & \sum T_{\widetilde{f_{i_1\ldots i_m}},x^{i_1},\ldots ,x^{i_m}}+\sum T_{\widetilde{g_{j_1\ldots j_{m-1}}},x^{j_1}\ldots ,x^{j_{m-1}},1}\\
                &   & \quad - \sum T_{\widetilde{f_{i_1\ldots i_m}}x^{i_m},x^{i_1},\ldots ,x^{i_{m-1}},1}
\end{eqnarray*}
Define,
\begin{eqnarray*}
&   & \begin{pmatrix}
 \phi & 0\\
 0 & \psi
\end{pmatrix}\, .\, (\omega_{m-1},\omega_m,\widetilde \omega_{m-1},\widetilde \omega_m)\\
& := & \Phi \left(\begin{pmatrix}
 \phi & 0\\
 0 & \psi
\end{pmatrix}\, .\, \Phi^{-1}(\omega_{m-1},\omega_m,\widetilde \omega_{m-1},\widetilde \omega_m) \right)\\
& = & \begin{cases}
        \begin{array}{lcl}
          \Phi \left(\begin{pmatrix}
 \phi & 0\\
 0 & \psi
\end{pmatrix}\, .\, \begin{pmatrix}
 \xi & 0\\
 0 & \widetilde \xi
\end{pmatrix}\, \,   \right)\, \, ;\, if\, \, m\, \, even\\
          \Phi \left(\begin{pmatrix}
 \phi & 0\\
 0 & \psi
\end{pmatrix}\, .\, \begin{pmatrix}
 0 & \xi\\
 \widetilde \xi & 0
\end{pmatrix}\, \,   \right)\, \, ;\, if\, \, m\, \, odd
        \end{array}
      \end{cases}\\
& = & (\phi \omega_{m-1},\phi \omega_m,\psi \widetilde \omega_{m-1},\psi \widetilde \omega_m)\, \, ;\, for\, \, both\, \, even\, \, and\, \, odd\, \, \, m\, .
\end{eqnarray*}  
and
\begin{eqnarray*}
&   & (\omega_{m-1},\omega_m,\widetilde \omega_{m-1},\widetilde \omega_m)\, .\, \begin{pmatrix}
 \phi & 0\\
 0 & \psi
\end{pmatrix}\\
& := & \Phi \left(\Phi^{-1}(\omega_{m-1},\omega_m,\widetilde \omega_{m-1},\widetilde \omega_m)\, .\, \begin{pmatrix}
 \phi & 0\\
 0 & \psi
\end{pmatrix} \right)\\
& = & \begin{cases}
        \begin{array}{lcl}
          \Phi \left(\begin{pmatrix}
 \xi & 0\\
 0 & \widetilde \xi
\end{pmatrix}\, .\, \begin{pmatrix}
 \phi & 0\\
 0 & \psi
\end{pmatrix}\, \right)\, \, ;\, if\, \, m\, \, even\\
          \Phi \left(\begin{pmatrix}
 0 & \xi\\
 \widetilde \xi & 0
\end{pmatrix}\, .\, \begin{pmatrix}
 \phi & 0\\
 0 & \psi
\end{pmatrix}\, \right)\, \, ;\, if\, \, m\, \, odd
        \end{array}
      \end{cases}\\
& = & \begin{cases}
        \begin{array}{lcl}
          \Phi \left(\begin{pmatrix}
 \xi\phi & 0\\
 0 & \widetilde \xi\psi
\end{pmatrix}\, \, \right)\, \, ;\, if\, \, m\, \, even\\
          \Phi \left(\begin{pmatrix}
 0 & \xi\psi\\
 \widetilde \xi \phi & 0
\end{pmatrix}\, \, \right)\, \, ;\, if\, \, m\, \, odd
        \end{array}
      \end{cases}
\end{eqnarray*}
where $\Phi$ is the map defined in Lemma $\,$\ref{bijection for m}~. Now
\begin{eqnarray*}
 \xi\phi & = & \sum T_{f_{i_1\ldots i_m},x^{i_1},\ldots ,x^{i_m}\phi}+\sum T_{g_{j_1\ldots j_{m-1}},x^{j_1},\ldots ,x^{j_{m-1}},\phi}\\
         &   & \quad - \sum T_{f_{i_1\ldots i_m}x^{i_m},x^{i_1},\ldots ,x^{i_{m-1}},\phi}
\end{eqnarray*}
and
\begin{eqnarray*}
 \widetilde \xi \psi & = & \sum T_{\widetilde{f_{i_1\ldots i_m}},x^{i_1},\ldots ,x^{i_m}\psi}+\sum T_{\widetilde{g_{j_1\ldots j_{m-1}}},x^{j_1},\ldots ,x^{j_{m-1}},\psi}\\
                     &   & \quad - \sum T_{\widetilde{f_{i_1\ldots i_m}}x^{i_m},x^{i_1},\ldots ,x^{i_{m-1}},\psi}
\end{eqnarray*}
So
\begin{eqnarray*}
&   & \Phi \left(\begin{pmatrix}
 \xi\phi & 0\\
 0 & \widetilde \xi\psi
\end{pmatrix} \right)\\
& = & (\, \sum f_{i_1\ldots i_m}x^{i_m}\phi dx^{i_1}\wedge \ldots \wedge dx^{i_{m-1}}-f_{i_1\ldots i_m}x^{i_m}\phi dx^{i_1}\wedge \ldots \wedge dx^{i_{m-1}}\\
&   & +\phi\omega_{m-1}\, \, ,\sum f_{i_1\ldots i_m} dx^{i_1}\wedge \ldots dx^{i_{m-1}}\wedge d(x^{i_m}\phi)+\omega_{m-1}\wedge d\phi\\
&   & -f_{i_1\ldots i_m}x^{i_m} dx^{i_1}\wedge \ldots dx^{i_{m-1}}\wedge d\phi\, \, ,\sum \widetilde{f_{i_1\ldots i_m}}x^{i_m}\psi dx^{i_1}\wedge \ldots \wedge dx^{i_{m-1}}\\
&   & -\widetilde{f_{i_1\ldots i_m}}x^{i_m}\psi dx^{i_1}\wedge \ldots \wedge dx^{i_{m-1}}+\psi \widetilde \omega_{m-1}\, \, ,\\
&   & \sum \widetilde{f_{i_1\ldots i_m}} dx^{i_1}\wedge \ldots \wedge dx^{i_{m-1}}\wedge d(x^{i_m}\psi)+\widetilde \omega_{m-1}\wedge d\psi\\
&   & -\widetilde{f_{i_1\ldots i_m}}x^{i_m}\psi dx^{i_1}\wedge \ldots \wedge dx^{i_{m-1}}\wedge d\psi\, )\\
& = & (\phi \omega_{m-1},\phi \omega_m-d\phi \wedge \omega_{m-1},\psi \widetilde \omega_{m-1},\psi \widetilde \omega_m-d\psi \wedge \widetilde \omega_{m-1})
\end{eqnarray*}
Similarly one can prove that
\begin{eqnarray*}
&   & \Phi \left(\begin{pmatrix}
 0 & \xi\psi\\
 \widetilde \xi\phi & 0
\end{pmatrix} \right)\\
& = & (\psi \omega_{m-1},\psi \omega_m+\omega_{m-1}\wedge d\psi,\phi \widetilde \omega_{m-1},\phi \widetilde \omega_m+\widetilde \omega_{m-1}\wedge d\phi)\\
& = & (\psi \omega_{m-1},\psi \omega_m+d\psi \wedge \omega_{m-1} ,\phi \widetilde \omega_{m-1},\phi \widetilde \omega_m+d\phi \wedge \widetilde \omega_{m-1}) 
\end{eqnarray*}
This is clearly a bimodule structure since it is induced by that on $\Omega_D^m\left(\widetilde{C^\infty(\mathbb{M})}\right)$.
\end{proof}
\medskip

\textbf{Notation~:} $\widetilde{\Omega_D^m} := \Omega^{m-1}\mathbb{M}\,\oplus\,\Omega^m\mathbb{M}\,\oplus\,\Omega^{m-1}\mathbb{M}\,\oplus\,\Omega^m\mathbb{M}\,$, $\,1\leq m\leq n\,$, untill the end this section,
where $\,\Omega^\bullet\mathbb{M}\,$ denotes the space of forms on $\mathbb{M}\,$.
\medskip

\begin{theorem}\label{bimodule iso}
$\Omega_D^m\left(\widetilde{C^\infty(\mathbb{M})}\right) \cong \widetilde{\Omega_D^m}\, $, for all $\,1\leq m\leq n\,$, and $\Omega_D^m\left(\widetilde{C^\infty(\mathbb{M})}\right)=\{0\}$ for $m>n$. This isomorphism is
a $\widetilde{C^\infty(\mathbb{M})}$-bimodule isomorphism.
\end{theorem}
\begin{proof}
We have for all $\,1\leq m\leq n\,$,
\begin{eqnarray*}
 \Omega_D^m\left(\widetilde{C^\infty(\mathbb{M})}\right) & \cong & \pi\left(\Omega^m(\widetilde{C^\infty(\mathbb{M})}\right)\quad\quad by\,\,Lemma\,\,\,\ref{denominator vanishing}\\
                                                         & \cong & \Omega^{m-1}\mathbb{M} \oplus \Omega^m\mathbb{M} \oplus \Omega^{m-1}\mathbb{M} \oplus \Omega^m\mathbb{M}\quad\quad by\,\,Lemma\,\,\ref{bijection for m}\,\,.
\end{eqnarray*}
Lemma $\,$\ref{higher forms vanish}$\,$ proves that $\Omega_D^m\left(\widetilde{C^\infty(\mathbb{M})}\right)=\{0\}$ for $m>n$. Finally Proposition $\,$\ref{bimodule action}$\,$ proves that this isomorphism is
$\widetilde{C^\infty(\mathbb{M})}$-bimodule isomorphism for all $\,1\leq m\leq n\,$. 
\end{proof}

Now we will turn $\,\widetilde{\Omega_D^\bullet}\,$ into a chain complex. To avoid confusion we denote the induced differential $\,d:\Omega_D^\bullet(\mathcal{A})\longrightarrow\Omega_D^{\bullet+1}(\mathcal{A})$ of diagram
$\,$\ref{induced differential of Connes}$\,$ by $\,\widetilde{d}\,$ in this section so that it should not be confused with the exterior differentiation $\,d\,$.
\begin{lemma}\label{action of differential}
The differential $\,\widetilde{d}:\Omega_D^m\left(\widetilde{C^\infty(\mathbb{M})}\right)\longrightarrow \Omega_D^{m+1}\left(\widetilde{C^\infty(\mathbb{M})}\right)$ of diagram $\,\ref{induced differential of Connes}\,$ has the
following action $\colon$
\begin{enumerate}
\item For $m\geq 1$ odd~,
\begin{center}
 $\widetilde{d}:\begin{pmatrix}
 0 & T_{a_0,\ldots ,a_m}\\
 T_{a_0^\prime,\ldots ,a_m^\prime} & 0
\end{pmatrix}\longmapsto \begin{pmatrix}
  T_{1,a_0^\prime,\ldots,a_m^\prime} + T_{a_0,\ldots,a_m,1} & 0\\
  0 & T_{1,a_0,\ldots,a_m} + T_{a_0^\prime,\ldots,a_m^\prime,1}
\end{pmatrix}$
\end{center}
\item For $m\geq 2$ even$\, $,
\begin{center}
 $\widetilde{d}:\begin{pmatrix}
  T_{a_0,\ldots ,a_m} & 0\\
  0 & T_{a_0^\prime,\ldots,a_m^\prime}
\end{pmatrix}\longmapsto \begin{pmatrix}
  0 & T_{1,a_0^\prime,\ldots,a_m^\prime} - T_{a_0,\ldots,a_m,1}\\
  T_{1,a_0,\ldots,a_m} - T_{a_0^\prime,\ldots,a_m^\prime,1} & 0 
\end{pmatrix}$
\end{center}

\end{enumerate}
\end{lemma}
\begin{proof}
We first note that
\begin{eqnarray}
 \left[D,\begin{pmatrix}
  0 & 0\\
  0 & 1
\end{pmatrix}\right] & = & \begin{pmatrix}
  0 & T_{1,1}\\
  -T_{1,1} & 0
\end{pmatrix}
\end{eqnarray}
\begin{eqnarray}
 \begin{pmatrix}
  0 & T_{1,1}\\
  -T_{1,1} & 0
\end{pmatrix}\begin{pmatrix}
  -a_1^\prime & 0\\
  0 & a_1
\end{pmatrix} & = & \begin{pmatrix}
  0 & T_{1,a_1}\\
  T_{1,a_1^\prime} & 0
\end{pmatrix}
\end{eqnarray}
\begin{eqnarray}
 \begin{pmatrix}
  a_0 & 0\\
  0 & a_0^\prime
\end{pmatrix}\begin{pmatrix}
  0 & T_{1,a_1}\\
  T_{1,a_1^\prime} & 0
\end{pmatrix} & = & \begin{pmatrix}
  0 & T_{a_0,a_1}\\
  T_{a_0^\prime,a_1^\prime} & 0
\end{pmatrix}
\end{eqnarray}
Hence combining these three we get,
\begin{eqnarray}\label{usefull eqn for m=1}
\begin{pmatrix}
  a_0 & 0\\
  0 & a_0^\prime
\end{pmatrix}\left[D,\begin{pmatrix}
  0 & 0\\
  0 & 1
\end{pmatrix}\right]\begin{pmatrix}
  -a_1^\prime & 0\\
  0 & a_1
\end{pmatrix} & = & \begin{pmatrix}
  0 & T_{a_0,a_1}\\
  T_{a_0^\prime,a_1^\prime} & 0
\end{pmatrix} 
\end{eqnarray}
\underline{Case 1~:} $\,\,$ Let $m\geq 3\, $ be odd. Observe that
\begin{eqnarray*}
&  & \begin{pmatrix}
  0 & T_{a_0,a_1,\ldots ,a_m}\\
  T_{a_0^\prime ,a_1^\prime,\ldots ,a_m^\prime} & 0
\end{pmatrix}\\
& = & \begin{pmatrix}
  0 & T_{a_0,a_1}\\
  T_{a_0^\prime ,a_1^\prime} & 0
\end{pmatrix}\prod_{i=2,i\, \, even}^m \left(\begin{pmatrix}
  0 & T_{1,a_i^\prime}\\
  T_{1,a_i} & 0
\end{pmatrix}\begin{pmatrix}
  0 & T_{1,a_{i+1}}\\
  T_{1,a_{i+1}^\prime} & 0
\end{pmatrix}\right)\\
& = & \begin{pmatrix}
  a_0 & 0\\
  0 & a_0^\prime
\end{pmatrix}\left[D,\begin{pmatrix}
  0 & 0\\
  0 & 1
\end{pmatrix}\right]\begin{pmatrix}
  -a_1^\prime & 0\\
  0 & a_1
\end{pmatrix}\bullet\\
&   & \prod_{i=2,i\, \, even}^m \left(\left\{\left[D,\begin{pmatrix}
  0 & 0\\
  0 & 1
\end{pmatrix}\right]\begin{pmatrix}
  -a_i & 0\\
  0 & a_i^\prime
\end{pmatrix}\right\} \left\{\left[D,\begin{pmatrix}
  0 & 0\\
  0 & 1
\end{pmatrix}\right]\begin{pmatrix}
  -a_{i+1}^\prime & 0\\
  0 & a_{i+1}
\end{pmatrix}\right\}\right)
\end{eqnarray*}
Consider the expression $\,\eta=x_0\prod_{i=1}^m\left(\bar{d}(b)x_i\right)\, $ where,
\begin{center}
 $x_0=\begin{pmatrix}
  a_0 & 0\\
  0 & a_0^\prime
\end{pmatrix}\, ;\, b=\begin{pmatrix}
  0 & 0\\
  0 & 1
\end{pmatrix}\, ;\, \bar{d}(y)=\left[D,\begin{pmatrix}
  y_{11} & y_{12}\\
  y_{21} & y_{22}
\end{pmatrix}\right]$
\end{center}
\begin{center}
 $x_i=\begin{pmatrix}
  -a_i^\prime & 0\\
  0 & a_i
\end{pmatrix}\, for\, 1\leq i\leq m\, ,odd\,\,\, ;\,\,\, x_j=\begin{pmatrix}
  -a_j & 0\\
  0 & a_j^\prime
\end{pmatrix}\, for\, 1\leq j\leq m\, ,even.$
\end{center}
One should note that $\bar{d}\circ \bar{d}(b)=0$ because $d^2=0$, $d$ being the exterior differentiation. Now for the differential $\, \, \widetilde{d}:\Omega_D^m\left(\widetilde{C^\infty(\mathbb{M})}\right)\longrightarrow
\Omega_D^{m+1}\left(\widetilde{C^\infty(\mathbb{M})}\right)\, $ of diagram $\,$\ref{induced differential of Connes}$\,$ we get,
\begin{eqnarray*}
 \widetilde{d}\eta & = & \bar{d}(x_0)\prod_{i=1}^m\{\bar{d}(b)x_i\} + x_0\widetilde{d}\left(\bar{d}(b)x_1\prod_{i=2}^m\{\bar{d}(b)x_i\}\right)\\
       & = & \bar{d}(x_0)\prod_{i=1}^m\{\bar{d}(b)x_i\} + \sum_{k=2}^{m}(-1)^{k-1}\prod_{j=0}^{k-2}\{x_j\bar{d}(b)\}\bar{d}(x_{k-1})\left(\prod_{i=k}^m\{\bar{d}(b)x_i\}\right)\\
       &   & +(-1)^m \left(\prod_{i=0}^{m-1}\{x_i\bar{d}(b)\}\right)\bar{d}(x_m)\\
       & = & \left[D,\begin{pmatrix}
  a_0 & 0\\
  0 & a_0^\prime
\end{pmatrix} \right]\left[D,\begin{pmatrix}
  0 & 0\\
  0 & 1
\end{pmatrix} \right]\begin{pmatrix}
  -a_1^\prime & 0\\
  0 & a_1
\end{pmatrix} \prod_{i=2,i\, \, even}^m \left(\begin{pmatrix}
  0 & T_{1,a_i^\prime}\\
  T_{1,a_i} & 0
\end{pmatrix}\begin{pmatrix}
  0 & T_{1,a_{i+1}}\\
  T_{1,a_{i+1}^\prime} & 0
\end{pmatrix}\right)\\
&   & + \sum_{k=2}^{m}(-1)^{k-1}\prod_{j=0}^{k-2}\{x_j\bar{d}(b)\}\bar{d}(x_{k-1})\left(\prod_{i=k}^m\{\bar{d}(b)x_i\}\right)\\
       &   & +(-1)^m\begin{pmatrix}
  a_0 & 0\\
  0 & a_0^\prime
\end{pmatrix}\left[D,\begin{pmatrix}
  0 & 0\\
  0 & 1
\end{pmatrix} \right]\prod_{i=1,i\, \, odd}^{m-1}\textbf{\{}\begin{pmatrix}
  -a_i^\prime & 0\\
  0 & a_i
\end{pmatrix}\left[D,\begin{pmatrix}
  0 & 0\\
  0 & 1
\end{pmatrix} \right]\\
       &   & \begin{pmatrix}
  -a_{i+1} & 0\\
  0 & a_{i+1}^\prime
\end{pmatrix}\left[D,\begin{pmatrix}
  0 & 0\\
  0 & 1
\end{pmatrix} \right]\textbf{\}}\left[D,\begin{pmatrix}
  -a_m^\prime & 0\\
  0 & a_m
\end{pmatrix} \right]\\
\end{eqnarray*}
Now it is straightforward computation to observe that
\begin{eqnarray*}
 \prod_{i=k}^m\{\bar{d}(b)x_i\} & = & \prod_{i=k}^m\left(\begin{pmatrix}
                                                               0 & T_{1,1}\\
                                                              -T_{1,1} & 0
                                                              \end{pmatrix}\,x_i\right)\\
                                     & = & \begin{cases}
                                            \begin{pmatrix}
                                             T_{1,a_k^\prime,\ldots,a_m^\prime} & 0\\
                                             0 & T_{1,a_k,\ldots,a_m}
                                             \end{pmatrix}\quad;\,\,\,if\,\,k\,\,even\\
                                             \begin{pmatrix}
                                              0 & T_{1,a_k,\ldots,a_m}\\
                                              T_{1,a_k^\prime,\ldots,a_m^\prime} & 0
                                              \end{pmatrix}\quad;\,\,\,if\,\,k\,\,odd
                                             \end{cases}
\end{eqnarray*}
and
\begin{eqnarray*}
 \prod_{j=0}^{k-2} \{x_j\bar{d}(b)\}\bar{d}(x_{k-1}) & = & \begin{cases}
                                            \begin{pmatrix}
                                             -T_{a_0,\ldots,a_{k-1},1} & 0\\
                                             0 & -T_{a_0^\prime,\ldots,a_{k-1}^\prime,1}
                                             \end{pmatrix}\quad;\,\,\,if\,\,k\,\,even\\
                                             \begin{pmatrix}
                                              0 & -T_{a_0,\ldots,a_{k-1},1}\\
                                              -T_{a_0^\prime,\ldots,a_{k-1}^\prime,1} & 0
                                              \end{pmatrix}\quad;\,\,\,if\,\,k\,\,odd
                                             \end{cases}
\end{eqnarray*}
The fact $d^2=0$ will now ensure that only the first and last term in the expression for $\widetilde{d}\eta$ survive. Hence,
\begin{eqnarray*}
\widetilde{d}\eta & = & \begin{pmatrix}
  (T_{1,a_0^\prime}-T_{a_0,1})T_{1,a_1^\prime,\ldots,a_m^\prime} & 0\\
  0 & (T_{1,a_0}-T_{a_0^\prime,1})T_{1,a_1,\ldots,a_m}\\
\end{pmatrix} + (-1)^m\begin{pmatrix}
  -T_{a_0,\ldots,a_m,1} & 0\\
  0 & -T_{a_0^\prime,\ldots,a_m^\prime,1}\\
\end{pmatrix}\\
       & = & \begin{pmatrix}
  T_{1,a_0^\prime,\ldots,a_m^\prime} + T_{a_0,\ldots,a_m,1} & 0\\
  0 & T_{1,a_0,\ldots,a_m} + T_{a_0^\prime,\ldots,a_m^\prime,1}\\
\end{pmatrix}\, .
\end{eqnarray*}

\underline{Case 2~:} $\,\,$ Let $m\, $ be even.

One can prove in exact similar manner like the `odd' case. The only difference in this case is a negative sign and it appears because of the presence of $(-1)^m$ at the last term in the expression for $\widetilde{d}\eta$.

\underline{Case 3~:} $\,\,$ Let $m=1\, $. Recall from equation ($\,$\ref{usefull eqn for m=1}$\,$),
\begin{center}
$\begin{pmatrix}
  0 & T_{a_0,a_1}\\
  T_{a_0^\prime,a_1^\prime} & 0
\end{pmatrix} = \begin{pmatrix}
  a_0 & 0\\
  0 & a_0^\prime
\end{pmatrix}\left[D,\begin{pmatrix}
  0 & 0\\
  0 & 1
\end{pmatrix}\right]\begin{pmatrix}
  -a_1^\prime & 0\\
  0 & a_1
\end{pmatrix}$
\end{center}
and hence,
\begin{center}
 $\widetilde{d}:\begin{pmatrix}
  0 & T_{a_0,a_1}\\
  T_{a_0^\prime,a_1^\prime} & 0
\end{pmatrix}\longmapsto \begin{pmatrix}
  T_{1,a_0^\prime,a_1^\prime}+T_{a_0,a_1,1} & 0\\
  0 & T_{1,a_0,a_1}+T_{a_0^\prime,a_1^\prime,1} 
\end{pmatrix}$
\end{center}
\end{proof}

Using the isomorphism in Theorem \ref{bimodule iso} we can transfer the differential $\widetilde{d}:\Omega^\bullet_D \left(\widetilde{C^\infty(\mathbb{M})}\right)\longrightarrow\Omega^{\bullet+1}_D \left(\widetilde{C^\infty
(\mathbb{M})}\right)$ to the differential $\delta:\widetilde{\Omega_D^\bullet}\longrightarrow \widetilde{\Omega_D^{\bullet+1}}$. This will turn $\widetilde{\Omega_D^\bullet}$ into a chain complex and then we will be able to
compute the cohomologies of the complex $\left(\widetilde{\Omega_D^\bullet}\,,\,\delta\right)$.

\begin{proposition}\label{the differential}
For $1\leq m\leq n\, $, the map
\begin{center}
 $\delta:\widetilde{\Omega_D^m}\longrightarrow \widetilde{\Omega_D^{m+1}}$
\end{center}
\begin{center}
 $(\omega_{m-1},\omega_m,\widetilde \omega_{m-1},\widetilde \omega_m)\longmapsto \left(d\widetilde \omega_{m-1}+(-1)^m(\widetilde \omega_m-\omega_m)\, ,\, d\widetilde \omega_m\, ,\, d\omega_{m-1}+(-1)^m(\omega_m
-\widetilde \omega_m)\, ,\, d\omega_m\right)$
\end{center}
makes the following diagram
\begin{center}
\begin{tikzpicture}[node distance=3cm,auto]
\node (Up)[label=above:$\widetilde{d}$]{};
\node (A)[node distance=2cm,left of=Up]{$\Omega^m_D \left(\widetilde{C^\infty(\mathbb{M})}\right)$};
\node (B)[node distance=1.7cm,right of=Up]{$\Omega^{m+1}_D \left(\widetilde{C^\infty(\mathbb{M})}\right)$};
\node (Down)[node distance=1.7cm,below of=Up, label=below:$\delta$]{};
\node(C)[node distance=2cm,left of=Down]{$\widetilde{\Omega^m_D}$};
\node(D)[node distance=1.7cm,right of=Down]{$\widetilde{\Omega^{m+1}_D}$};
\draw[->](A) to (B);
\draw[->](C) to (D);
\draw[->](B)to node{{ $\cong$}}(D);
\draw[->](A)to node[swap]{{ $\cong$}}(C);
\end{tikzpicture} 
\end{center} commutative.
\end{proposition}
\begin{proof}
For $1\leq m\leq n$ take $(\omega_{m-1},\omega_m,\widetilde \omega_{m-1},\widetilde \omega_m) \in \widetilde{\Omega_D^m}$. In terms of local co-ordinates
\begin{center}
 $\omega_{m-1}=\sum_{j_1<\ldots <j_{m-1}} g_{j_1\ldots j_{m-1}}dx^{j_1}\wedge \ldots \wedge dx^{j_{m-1}}$
\end{center}
\begin{center}
 $\omega_m=\sum_{i_1<\ldots <i_m} f_{i_1\ldots i_m}dx^{i_1}\wedge \ldots \wedge dx^{i_m}$
\end{center}
\begin{center}
 $\widetilde \omega_{m-1}=\sum_{j_1<\ldots <j_{m-1}} \widetilde{g_{j_1\ldots j_{m-1}}}dx^{j_1}\wedge \ldots \wedge dx^{j_{m-1}}$
\end{center}
\begin{center}
 $\widetilde \omega_m=\sum_{i_1<\ldots <i_m} \widetilde{f_{i_1\ldots i_m}}dx^{i_1}\wedge \ldots \wedge dx^{i_m}$
\end{center}
Using Lemma $\,$\ref{bijection for m}$\,$ we see that, isomorphic image of this element in $\Omega_D^m\left(\widetilde{C^\infty(\mathbb{M})}\right)$ is
\begin{center}$
\begin{cases}
        \begin{array}{lcl}
          \begin{pmatrix}
 \xi & 0\\
 0 & \widetilde \xi
\end{pmatrix}\, \, ;\, if\, \, m\, \, even\\
          \begin{pmatrix}
 0 & \xi\\
 \widetilde \xi & 0
\end{pmatrix}\, \, ;\, if\, \, m\, \, odd
        \end{array}
      \end{cases}$
\end{center}
where,
\begin{eqnarray*}
 \xi & = & \sum T_{f_{i_1\ldots i_m},x^{i_1},\ldots,x^{i_m}}-T_{f_{i_1\ldots i_m}x^{i_m},x^{i_1},\ldots,x^{i_{m-1}},1}\\
     &   & \quad +\sum T_{g_{j_1\ldots j_{m-1}},x^{j_1},\ldots,x^{j_{m-1}},1}
\end{eqnarray*}
and
\begin{eqnarray*}
 \widetilde \xi & = & \sum T_{\widetilde{f_{i_1\ldots i_m}},x^{i_1},\ldots,x^{i_m}}-T_{\widetilde{f_{i_1\ldots i_m}}x^{i_m},x^{i_1},\ldots,x^{i_{m-1}},1}\\
     &   & \quad +\sum T_{\widetilde{g_{j_1\ldots j_{m-1}}},x^{j_1},\ldots,x^{j_{m-1}},1}
\end{eqnarray*}
By Lemma $\,$\ref{action of differential}$\,$ we see that the differential $\,\widetilde{d}:\Omega_D^m\left(\widetilde{C^\infty(\mathbb{M})}\right)\longrightarrow \Omega_D^{m+1}\left(\widetilde{C^\infty(\mathbb{M})}\right)$
sends this element to
\begin{center}$
\begin{cases}
        \begin{array}{lcl}
          \begin{pmatrix}
 d_{\widetilde \xi}+\xi\,T_{1,1} & 0\\
 0 & d_\xi+\widetilde{\xi}\,T_{1,1}
\end{pmatrix}\, \, ;\, if\, \, m\, \, odd\\
          \begin{pmatrix}
 0 & d_{\widetilde \xi}-\xi\,T_{1,1}\\
  d_\xi-\widetilde{\xi}\,T_{1,1} & 0
\end{pmatrix}\, \, ;\, if\, \, m\, \, even
        \end{array}
      \end{cases}$
\end{center}
where
\begin{eqnarray*}
 d_\xi & = & \sum T_{1,f_{i_1\ldots i_m},x^{i_1},\ldots,x^{i_m}}-T_{1,f_{i_1\ldots i_m}x^{i_m},x^{i_1},\ldots,x^{i_{m-1}},1}\\
     &   & \quad +\sum T_{1,g_{j_1\ldots j_{m-1}},x^{j_1},\ldots,x^{j_{m-1}},1}
\end{eqnarray*}
and
\begin{eqnarray*}
 d_{\widetilde \xi} & = & \sum T_{1,\widetilde{f_{i_1\ldots i_m}},x^{i_1},\ldots,x^{i_m}}-T_{1,\widetilde{f_{i_1\ldots i_m}}x^{i_m},x^{i_1},\ldots,x^{i_{m-1}},1}\\
     &   & \quad +\sum T_{1,\widetilde{g_{j_1\ldots j_{m-1}}},x^{j_1},\ldots,x^{j_{m-1}},1}
\end{eqnarray*}
Isomophic image of this element in $\, \widetilde{\Omega_D^{m+1}}\, $, under the map $\Phi$ of Lemma $\, $\ref{bijection for m}$\, $, is
\begin{center}$
\begin{cases}
        \begin{array}{lcl}
          (d\widetilde \omega_{m-1}-\widetilde \omega_m+\omega_m\, ,\, d\widetilde \omega_m\, ,\, d\omega_{m-1}-\omega_m+\widetilde \omega_m\, ,\, d\omega_m)\, \, ;\, if\, \, m\, \, odd\\
          (d\widetilde \omega_{m-1}+\widetilde \omega_m-\omega_m\, ,\, d\widetilde \omega_m\, ,\, d\omega_{m-1}+\omega_m-\widetilde \omega_m\, ,\, d\omega_m)\, \, ;\, if\, \, m\, \, even
        \end{array}
      \end{cases}$
\end{center}
i,e. $\, \left(d\widetilde \omega_{m-1}+(-1)^m(\widetilde \omega_m-\omega_m)\, ,\, d\widetilde \omega_m\, ,\, d\omega_{m-1}+(-1)^m(\omega_m-\widetilde \omega_m)\, ,\, d\omega_m\right)\, $.
\end{proof}

\begin{remark}
Notice that $\delta=\Phi\circ\widetilde{d}\circ \Phi^{-1}\, $, and hence $\delta^2=0$. Thus $\left(\widetilde{\Omega_D^\bullet}\, ,\delta\right)$ is a chain complex. Furthermore, the graded algebra structure on 
$\,\Omega_D^\bullet \left(\widetilde{C^\infty(\mathbb{M})}\right)$ will induce the same on $\left(\widetilde{\Omega_D^\bullet}\, ,\delta\right)$ through the commutative diagram of Proposition $\, \ref{the differential}\, $.
So we get $\left(\Omega_D^\bullet(\widetilde{C^\infty(\mathbb{M})}),\widetilde{d}\,\right) \cong \left(\widetilde{\Omega_D^\bullet}\, ,\delta\right)$ as differential graded algebras and Theorem $\,\ref{bimodule iso}$ gives
$\,\widetilde{C^\infty(\mathbb{M})}$-bimodule isomorphism at each term of these chain complexes.
\end{remark}

\begin{theorem}
The cohomologies $\widetilde{H^\bullet(\mathbb{M})}$ of the chain complex $\left(\widetilde{\Omega_D^\bullet}\, ,\delta\right)$ are given by,
$$\widetilde{H^m(\mathbb{M})} \cong H^{m-1}(\mathbb{M})\oplus H^m(\mathbb{M})\,;\,\, for\,\,0\leq m\leq dim(\mathbb{M})\,,$$
where $H^\bullet(\mathbb{M})$ denotes the de-Rham cohomologies of $\, \mathbb{M}$.
\end{theorem}
\begin{proof}
$(1)\,\,$ Let $\,m=0$. Recall that for $\,\begin{pmatrix}
                   f & 0\\
                   0 & g
                  \end{pmatrix}\in \widetilde{C^\infty(\mathbb{M})}$, $$\left[D,\begin{pmatrix}
                   f & 0\\
                   0 & g
                  \end{pmatrix}\right]=\begin{pmatrix}
                   0 & T_{1,g}-T_{f,1}\\
                   T_{1,f}-T_{g,1} & 0
                  \end{pmatrix}\,.$$ The isomorphism of Lemma $\,$\ref{bijection for m}$\,$ sends this element to $(g-f,dg,f-g,df)$. Hence
\begin{eqnarray*}
 \widetilde{H^0(\mathbb{M})} & = & \textbf{\{}\begin{pmatrix}
                   f & 0\\
                   0 & f
                  \end{pmatrix} \colon df=0\, ,\, f\in C^\infty(\mathbb{M})\textbf{\}}\\
                             & \cong & H^0(\mathbb{M})\,.
\end{eqnarray*}

$(2)\,\,$ Let $\,1\leq m\leq dim(\mathbb{M})\, $. Consider $\,\delta^{m-1}:\widetilde{\Omega_D^{m-1}}\longrightarrow\widetilde{\Omega_D^m}\,$ and $\,\delta^m:\widetilde{\Omega_D^m}\longrightarrow\widetilde{\Omega_D^{m+1}}$. Then
\begin{eqnarray}\label{equation 3.16}
 \delta^{m-1}(v_{m-2},v_{m-1},\widetilde{v_{m-2}},\widetilde{v_{m-1}}) & = & (d\widetilde{v_{m-2}}+(-1)^{m-1}(\widetilde{v_{m-1}}-v_{m-1}),d\widetilde{v_{m-1}},
\end{eqnarray}
$$\quad\quad\quad\quad\quad\quad\quad\quad\quad\quad\quad\quad\quad\quad\quad\quad\quad dv_{m-2}+(-1)^{m-1}(v_{m-1}-\widetilde{v_{m-1}}),dv_{m-1})$$
for all $(v_{m-2},v_{m-1},\widetilde{v_{m-2}},\widetilde{v_{m-1}})\in \widetilde{\Omega_D^{m-1}}$. Let $\zeta=(w_{m-1},w_m,\widetilde{w_{m-1}},\widetilde{w_m})\in Ker(\delta^m)$. Then we have the following
\begin{eqnarray}\label{element of kernel of delta^m}
\begin{cases}
 \begin{array}{lcl}
  d(w_m)=0\\
  d(\widetilde{w_m})=0
 \end{array}\quad ;\quad
 \begin{array}{lcl}
  d(\widetilde{w_{m-1}})+(-1)^m(\widetilde{w_m}-w_m)=0\\
  d(w_{m-1})+(-1)^m(w_m-\widetilde{w_m})=0
 \end{array}
\end{cases}
\end{eqnarray}
Define $$\Psi:\frac{Ker(\delta^m)}{Im(\delta^{m-1})}\longrightarrow H^m(\mathbb{M})\oplus H^{m-1}(\mathbb{M})$$ $$\quad\quad\quad\quad\quad\quad[\zeta]\longmapsto \left([w_m+\widetilde{w_m}],[w_{m-1}+\widetilde{w_{m-1}}]\right).$$ This
map is well-defined (because of equation ($\,$\ref{element of kernel of delta^m}$\,$)) and linear. Now define $$\Phi:H^m(\mathbb{M})\oplus H^{m-1}(\mathbb{M})\longrightarrow \frac{Ker(\delta^m)}{Im(\delta^{m-1})}$$
$$\quad\quad\quad\quad\quad\quad\quad\quad\quad\quad\quad\quad([v_m],[v_{m-1}])\longmapsto \left[\left(\frac{1}{2}v_{m-1},\frac{1}{2}v_m,\frac{1}{2}v_{m-1},\frac{1}{2}v_m\right)\right].$$ Using equation
($\,$\ref{equation 3.16}$\,$) one can check that $\Phi$ is well-defined and
linear. Now observe that $\,\Psi\circ \Phi=Id$, and $$\Phi\circ\Psi\,([\zeta])=\left[\left(\frac{1}{2}(w_{m-1}+\widetilde{w_{m-1}}),\frac{1}{2}(w_m+\widetilde{w_m}),\frac{1}{2}(w_{m-1}+\widetilde{w_{m-1}}),\frac{1}{2}(w_m+
\widetilde{w_m})\right)\right].$$ If we can show that $$\xi=\left(\frac{1}{2}(\widetilde{w_{m-1}}-w_{m-1}),\frac{1}{2}(\widetilde{w_m}-w_m),\frac{1}{2}(w_{m-1}-\widetilde{w_{m-1}}),\frac{1}{2}(w_m-\widetilde{w_m})\right)
\in Im(\delta^{m-1})\,,$$ then $\,\Phi\circ \Psi\,$ will also be identity. Observe that $$\delta^{m-1}\left(0,\frac{(-1)^{m+1}}{4}(w_{m-1}-\widetilde{w_{m-1}}),0,\frac{(-1)^{m+1}}{4}(\widetilde{w_{m-1}}-w_{m-1})\right)=
\xi$$ using equation ($\,$\ref{element of kernel of delta^m}$\,$), and hence $(2)$ follows.
\end{proof}

\medskip


\section{Computation for the Noncommutative Torus}
In this section our objective is to show that the functor $\mathcal{F}\circ\mathcal{G}$ is not trivial for the case of noncommutative torus, one of the most fundamental and widely studied example in noncommutative geometry.
We recall the definition of noncommutative torus from (\cite{8}). Let $\theta$ be a real number. Denote by $\mathcal{A}_\varTheta$,  the universal $C^*$-algebra generated by unitaries 
$\, U, V\,$ satisfying $UV = e^{-2\pi i\theta} VU$. Throughout this section $i$ will stand for $\sqrt{-1}$. On $\mathcal{A}_\varTheta$, the Lie group $G=\mathbb{T}^2$ acts as follows:
\begin{center}
$\alpha_{(z_1,z_2)}(U) =z_1U\, \, $ and$\, \, \, \alpha_{(z_1,z_2)}(V) =z_2V\, $.
\end{center}
The smooth subalgebra of $\mathcal{A}_\varTheta$, is given by
\begin{eqnarray*}
\mathcal{A}_\varTheta^\infty := \textbf{\{}\sum \, a_{r_1,r_2}\, U^{r_1}V^{r_2} : \{ a_{r_1,r_2} \} \in \mathbb{S} (\mathbb{Z}^{2})\, , \, r_1,r_2 \in \mathbb{Z}\textbf{\}}
\end{eqnarray*}
where $\, \mathbb{S}(\mathbb{Z}^{2})$ denotes vector space of multisequences $(a_{r_1,r_2})$ that decay faster than the inverse of any polynomial in $\textbf{r} = (r_1,r_2)$. This subalgebra is equipped with a unique
$G$-invariant tracial state, given by $\tau(a) = a_{0,0}\,$. The Hilbert space obtained by applying the G.N.S. construction to $\tau$ can be identified with $\ell^2(\mathbb{Z}^2)$ (\cite{8}) and we have
$\mathcal{A}_\varTheta^\infty\subseteq\ell^2(\mathbb{Z}^2)$ as subspace. We have the following derivations
acting on $\mathcal{A}_\varTheta^\infty\,,$ $$\widetilde{\delta}_j(\sum_{r_1,r_2}a_{r_1,r_2}U^{r_1}V^{r_2}):=\sqrt{-1}\sum_{r_1,r_2}r_ja_{r_1,r_2}U^{r_1}V^{r_2}\quad for\,\,j=1,2\,.$$ Let $\,\delta_j:=-\sqrt{-1}\,\widetilde
{\delta}_j\,,j=1,2$. It is known that (\cite{3}) $$\left(\mathcal{A}_\varTheta^\infty\,,\,\ell^2(\mathbb{Z}^2)\otimes\mathbb{C}^2\,,\,D:=\begin{pmatrix}
 0 & \delta_1-i\delta_2\\
 \delta_1+i\delta_2 & 0
\end{pmatrix}\,,\,\gamma:=\begin{pmatrix}
 1 & 0\\
 0 & -1
\end{pmatrix}\right)$$ forms an even spectral triple on $\mathcal{A}_\varTheta^\infty$. Let $\mathcal{H}^\infty:=\bigcap_{k\geq 1}\mathcal{D}om(D^k)$. Then $\mathcal{H}^\infty=\mathcal{A}_\varTheta^\infty\otimes\mathbb{C}^2$.

In this section our candidate for the even algebraic spectral triple is the following quadruple $$\mathcal{E}:=\,\left(\mathcal{A}_\varTheta^\infty\,,\,\mathcal{H}^\infty\otimes\mathbb{C}^2\,,\,D:=\begin{pmatrix}
 0 & \delta_1-i\delta_2\\
 \delta_1+i\delta_2 & 0
\end{pmatrix}\,,\,\gamma:=\begin{pmatrix}
 1 & 0\\
 0 & -1
\end{pmatrix}\right)\,.$$ Since $\gamma\notin\pi(C^\infty(\mathbb{M}))$ we first apply the functor $\,\mathcal{G}\,$ of Proposition $\,$\ref{our construction}$\,$ and then compute $\,\mathcal{F}\circ\mathcal{G}\,$ along with the
associated cohomologies. We will only work with the smooth subalgebra $\mathcal{A}_\varTheta^\infty$ and hence denote it by $\mathcal{A}_\varTheta$ for notational brevity. Note that, $$\delta_1(U)=U\,,\,\delta_1(V)=0\,,\,\delta_2(U)=0\,,\,
\delta_2(V)=V\,.$$ We denote $\,d:=\delta_1-i\delta_2\,$ and $\,d^*:=\delta_1+i\delta_2\,$. Hence, $$d(U)=U\,,\,d^*(U)=U\,,\,d(V)=-iV\,,\,d^*(V)=\,iV\,,$$ $$d(U^*)=-U^*\,,\,d^*(U^*)=-U^*\,,\,d(V^*)=iV^*\,,\,d^*(V^*)=-iV^*\,.$$
\medskip

\textbf{Notation~:} $\widetilde{\mathcal{A}_\varTheta}=\mathcal{G}(\mathcal{E})\, \,$ throughout this section where $\,\mathcal{G}\,$ is as defined in Proposition $\,$\ref{our construction}$\,$.
\medskip

Note that $J_0^0\left(\widetilde{\mathcal{A}_\varTheta}\right)=\{0\}$ in this case. Now observe that
\begin{eqnarray}\label{torus 1 form}
 \left[D,\begin{pmatrix}
 a & 0\\
 0 & b
\end{pmatrix}\right] & = & \begin{pmatrix}
 0 & db-ad\\
 d^*a-bd^* & 0
\end{pmatrix}\, ,
\end{eqnarray}
and hence each element of $\pi\left(\Omega^1(\widetilde{\mathcal{A}_\varTheta})\right)$ is linear span of following elements~:
\begin{center}
 $\begin{pmatrix}
 0 & cde\\
 c^\prime d^*e^\prime & 0
\end{pmatrix}\quad$ such that $\, c,e,c^\prime,e^\prime \in  \mathcal{A}_\varTheta$.
\end{center}
For $\,b,c\in \mathcal{A}_\varTheta$ consider the linear operator
\begin{center}
 $cdb:\mathcal{A}_\varTheta \longrightarrow \mathcal{A}_\varTheta$
\end{center}
\begin{center}
 $\quad\quad\quad\,\,\, e\longmapsto cd(be)$
\end{center}
Let $\mathcal{M}_1:=span\{cdb:\mathcal{A}_\varTheta \longrightarrow \mathcal{A}_\varTheta \colon b,c\in \mathcal{A}_\varTheta\}$. Then $\mathcal{M}_1$ is a $\mathbb{C}\,$-vector space and using equation $\,$\ref{torus 1 form}
$\,$ we see that $\pi\left(\Omega^1(\widetilde{\mathcal{A}_\varTheta})\right)\subseteq\mathcal{M}_1\oplus\mathcal{M}_1$. Now the following equality
$$\begin{pmatrix}
 0 & cde\\
 c^\prime d^*e^\prime & 0
\end{pmatrix}=\begin{pmatrix}
 c & 0\\
 0 & -c^\prime
\end{pmatrix}\left[D,\begin{pmatrix}
 0 & 0\\
 0 & 1
\end{pmatrix}\right]\begin{pmatrix}
 e^\prime & 0\\
 0 & e
\end{pmatrix}$$ proves that $\,\pi\left(\Omega^1(\widetilde{\mathcal{A}_\varTheta})\right)=\mathcal{M}_1\oplus\mathcal{M}_1$.

\begin{lemma}\label{subspaces}
Let $\mathbb{V}$ be the vector space of linear endomorphisms acting on $\mathcal{A}_\varTheta$. Let $M_\xi$ denotes multiplication by $\xi$. The vector subspaces $\{M_{\sum c_id(b_i)}:\mathcal{A}_\varTheta\longrightarrow
\mathcal{A}_\varTheta\,\colon\,c_i,b_i\in\mathcal{A}_\varTheta\}$ and $\{M_{e}\circ d:\mathcal{A}_\varTheta\longrightarrow\mathcal{A}_\varTheta\,\colon\,e\in\mathcal{A}_\varTheta\}$ of $\,\mathbb{V}\,$ has trivial
intersection and $\mathcal{M}_1\subseteq\{M_{cd(b)}:\mathcal{A}_\varTheta \longrightarrow \mathcal{A}_\varTheta\}\bigoplus \{M_{cb}\circ d:\mathcal{A}_\varTheta \longrightarrow \mathcal{A}_\varTheta\}\,$.
\end{lemma}
\begin{proof}
 Observe that $cdb(e)= (M_{cd(b)}+M_{cb}\circ d)(e)$ for any $e\in \mathcal{A}_\varTheta$. Since $d(1)=0$ we have the direct sum.
\end{proof}

Let $T_{a,b}=(ad(b),ab)\in \mathcal{A}_\varTheta\oplus \mathcal{A}_\varTheta$ and $\widetilde{T_{a,b}}=(ad^*(b),ab)\in \mathcal{A}_\varTheta\oplus \mathcal{A}_\varTheta$. Define
\begin{center}
 $\quad\quad\quad\quad\Phi:\pi\left(\Omega^1(\widetilde{\mathcal{A}_\varTheta})\right)\longrightarrow \mathcal{A}_\varTheta\oplus \mathcal{A}_\varTheta \oplus \mathcal{A}_\varTheta\oplus \mathcal{A}_\varTheta$
\end{center}
\begin{center}
 $\begin{pmatrix}
 0 & adb\\
 a^\prime d^*b^\prime & 0
\end{pmatrix}\,\longmapsto (T_{a,b}\,,\,\widetilde{T_{a^\prime,b^\prime}})$
\end{center}

\begin{lemma}\label{bijection for torus}
 $\Phi$ is a linear bijection.
\end{lemma}
\begin{proof}
 To prove $\Phi$ is well-defined, let $\sum a_idb_i=0$. Acting it on $1\in \mathcal{A}_\varTheta$ and $U\in \mathcal{A}_\varTheta$ respectively, we see that both $\sum a_id(b_i)$ and $\sum a_ib_i$ are zero. Similarly for
the case of $\sum a_i^\prime d^*b_i^\prime=0$. This proves well-definedness and Lemma $\,$\ref{subspaces}$\,$ proves injectivity. To see surjectivity, observe that 
\begin{center}
$\begin{pmatrix}
 0 & aU^*dU+bd1-ad1\\
 a^\prime U^*d^*U+b^\prime d^*1-a^\prime d^*1 & 0
\end{pmatrix}\xrightarrow{\Phi} (a,b,a^\prime,b^\prime)$.
\end{center}
\end{proof}

\begin{proposition}\label{bimodule action for torus I}
 $\mathcal{A}_\varTheta\otimes \mathbb{C}^4$ is a $\widetilde{\mathcal{A}_\varTheta}$-bimodule where the module action is specified by
\begin{eqnarray*}
&   & \begin{pmatrix}
 f & 0\\
 0 & g
\end{pmatrix}.(a,b,a^\prime,b^\prime).\begin{pmatrix}
 f^\prime & 0\\
 0 & g^\prime
\end{pmatrix}\\
& := & \left(fag^\prime+fbd(g^\prime),fbg^\prime,ga^\prime f^\prime+gb^\prime d^*(f^\prime),gb^\prime f^\prime\right).
\end{eqnarray*}
\end{proposition}
\begin{proof}
If we define
\begin{center}
$\begin{pmatrix}
 f & 0\\
 0 & g
\end{pmatrix}.(a,b,a^\prime,b^\prime):=\Phi\left(\begin{pmatrix}
 f & 0\\
 0 & g
\end{pmatrix}.\Phi^{-1}(a,b,a^\prime,b^\prime)\right)\,,$
\end{center}
where $\Phi$ is in Lemma $\,$\ref{bijection for torus}$\,$, then it is clearly a left module structure induced by that on $\Omega_D^1(\widetilde{\mathcal{A}_\varTheta})$. Now one can check that
\begin{center}
$\begin{pmatrix}
 f & 0\\
 0 & g
\end{pmatrix}.(a,b,a^\prime,b^\prime)=(fa,fb,ga^\prime,gb^\prime)$
\end{center}
Similarly for the right module structure, we define
\begin{center}
$(a,b,a^\prime,b^\prime).\begin{pmatrix}
 f^\prime & 0\\
 0 & g^\prime
\end{pmatrix}:=\Phi\left(\Phi^{-1}(a,b,a^\prime,b^\prime).\begin{pmatrix}
 f^\prime & 0\\
 0 & g^\prime
\end{pmatrix}\right)$
\end{center}
and it equals to $\left(ag^\prime+bd(g^\prime),bg^\prime,a^\prime f^\prime+b^\prime d^*(f^\prime),b^\prime f^\prime\right)$.
\end{proof}

\begin{proposition}\label{space of 1-form for torus}
$\Omega_D^1(\widetilde{\mathcal{A}_\varTheta})\cong \mathcal{A}_\varTheta\otimes \mathbb{C}^4$ as $\widetilde{\mathcal{A}_\varTheta}$-bimodule.
\end{proposition}
\begin{proof}
The $\widetilde{\mathcal{A}_\varTheta}$-bimodule action on right hand side is given by Proposition $\, $\ref{bimodule action for torus I}$\, $ and $\Phi$ of Lemma $\, $\ref{bijection for torus}$\, $ becomes a bimodule isomorphism
under this action.
\end{proof}

Since elements of $\pi\left(\Omega^2(\widetilde{\mathcal{A}_\varTheta})\right)$ are linear sum of
$$\begin{pmatrix}
            a_0 & 0\\
            0 & b_0
           \end{pmatrix}\left[D,\begin{pmatrix}
            a_1 & 0\\
            0 & b_1
           \end{pmatrix}\right]\left[D,\begin{pmatrix}
            a_2 & 0\\
            0 & b_2
           \end{pmatrix}\right]\,,$$ they are of the form $\,\sum\begin{pmatrix}
 adbd^*c & 0\\
 0 & a^\prime d^*b^\prime dc^\prime
\end{pmatrix}$ for $a,b,a^\prime,b^\prime\in \mathcal{A}_\varTheta$. This shows that $\pi\left(\Omega^2(\widetilde{\mathcal{A}_\varTheta})\right)\subseteq\mathcal{M}_2\oplus\widetilde{\mathcal{M}_2}\,,$ where
$$\mathcal{M}_2:=span\{adbd^*c:\mathcal{A}_\varTheta\longrightarrow\mathcal{A}_\varTheta\}\,,$$ $$\quad\widetilde{\mathcal{M}_2}:=span\{a^\prime d^*b^\prime dc^\prime:\mathcal{A}_\varTheta\longrightarrow\mathcal{A}_\varTheta\}\,.$$
To see equality use equation $\,$\ref{torus 1 form}$\,$ and observe that
\begin{eqnarray*}
&  & \sum\begin{pmatrix}
 adbd^*c & 0\\
 0 & a^\prime d^*b^\prime dc^\prime
\end{pmatrix}\\
& = & \sum\begin{pmatrix}
            a & 0\\
            0 & -a^\prime
           \end{pmatrix}\left[D,\begin{pmatrix}
            0 & 0\\
            0 & 1
           \end{pmatrix}\right]\begin{pmatrix}
            b^\prime & 0\\
            0 & b
           \end{pmatrix}\left[D,\begin{pmatrix}
            0 & 0\\
            0 & 1
           \end{pmatrix}\right]\begin{pmatrix}
            -c & 0\\
            0 & c^\prime
           \end{pmatrix}\,. 
\end{eqnarray*}
Now consider the following linear operators
\begin{center}
$T_{a,b,c}:=adbd^*c:\mathcal{A}_\varTheta \longrightarrow \mathcal{A}_\varTheta$
\end{center}
\begin{center}
$\quad\quad\quad\quad\quad\quad\quad\quad\quad\quad\quad\quad e\longmapsto ad\left(bd^*(ce)\right)\,,$
\end{center}
\begin{center}
$\widetilde{T_{a^\prime,b^\prime,c^\prime}}:=a^\prime d^*b^\prime dc^\prime:\mathcal{A}_\varTheta \longrightarrow \mathcal{A}_\varTheta$
\end{center}
\begin{center}
$\quad\quad\quad\quad\quad\quad\quad\quad\quad\quad\quad\quad\quad\quad\,\, e\longmapsto a^\prime d^*\left(b^\prime d(c^\prime e)\right)\,.$
\end{center}
Then,
\begin{eqnarray}\label{definition of Tabc}
 \quad\quad\quad\begin{cases}
  \,T_{a,b,c}\,\,\,\,\,\equiv\, M_{ad(b)d^*(c)+abdd^*(c)} + M_{abd^*(c)}\circ d + M_{ad(bc)}\circ d^* + M_{abc}\circ d\circ d^*\\
  \,\widetilde{T_{a^\prime,b^\prime,c^\prime}}\equiv\, M_{a^\prime d^*(b^\prime)d(c^\prime)+a^\prime b^\prime d^*d(c^\prime)} + M_{a^\prime b^\prime d(c^\prime)}\circ d^* 
    + M_{a^\prime d^*(b^\prime c^\prime)}\circ d + M_{a^\prime b^\prime c^\prime}\circ d^*\circ d
 \end{cases}
\end{eqnarray}
where $M_\xi$ denotes multiplication by $\xi$.

\begin{lemma}\label{1st intersection}
 $\{M_f\circ d:\mathcal{A}_\varTheta \longrightarrow \mathcal{A}_\varTheta \}\bigcap \{M_g\circ d^*:\mathcal{A}_\varTheta \longrightarrow \mathcal{A}_\varTheta\}=\{0\}$
\end{lemma}
\begin{proof}
Let $\{e_k:k\in \mathbb{Z}^2\}$ denotes the standard orthonormal basis of $\ell^2(\mathbb{Z}^2)$. Here $\ell^2(\mathbb{Z}^2)$ represents the G.N.S. Hilbert space and $\mathcal{A}_\varTheta\subseteq \ell^2(\mathbb{Z}^2)$.
Any element from the intersection must satisfy
\begin{eqnarray*}
&   & \langle e_\alpha\, ,\, M_f\circ d(e_\beta) \rangle = \langle e_\alpha\, ,\, M_g\circ d^*(e_\beta) \rangle\quad\forall\,\alpha, \beta\in \mathbb{Z}^2\,.\\
& \Rightarrow & \langle\, \sum_k \widehat{f_k^*}e_{k+\alpha}\, ,\, d(e_\beta)\, \rangle = \langle\, \sum_k \widehat{g_k^*}e_{k+\alpha}\, ,\, d^*(e_\beta)\, \rangle\\
& \Rightarrow & \langle\, \sum_k \widehat{f_k^*}e_{k+\alpha}\, ,\, e_\beta\, \rangle(\beta_1-i\beta_2) = \langle\, \sum_k \widehat{g_k^*}e_{k+\alpha}\, ,\, e_\beta\, \rangle(\beta_1+i\beta_2)
\end{eqnarray*}
So,
\begin{center}
$\, \widehat{f_{\beta-\alpha}^*}(\beta_1-i\beta_2) = \widehat{g_{\beta-\alpha}^*}(\beta_1+i\beta_2)\, $
\end{center}
for all $\, \alpha=(\alpha_1,\alpha_2),\beta=(\beta_1,\beta_2)\in \mathbb{Z}^2\, $, i,e.
\begin{eqnarray}\label{equation I}
 \widehat{f_{\gamma}^*}(\alpha_1+\gamma_1-i\alpha_2-i\gamma_2) & = & \widehat{g_{\gamma}^*}(\alpha_1+\gamma_1+i\alpha_2+i\gamma_2)
\end{eqnarray}
where $\, \beta-\alpha=\gamma\in \mathbb{Z}^2\, $. In order to have nontrivial intersection, equation ($\, $\ref{equation I}$\, $) must have nontrivial solution for all $\alpha,\gamma\in \mathbb{Z}^2\, $. Let $\widehat{f_{\gamma}^*}=x$
and $\widehat{g_{\gamma}^*}=y$. We get
\begin{eqnarray}\label{eqn involving x,y I}
 x(1+\gamma_1-i-i\gamma_2) & = & y(1+\gamma_1+i+i\gamma_2)
\end{eqnarray}
\begin{eqnarray}\label{eqn involving x,y II}
 x(2+\gamma_1-2i-i\gamma_2) & = & y(2+\gamma_1+2i+i\gamma_2)
\end{eqnarray}
($\,$\ref{eqn involving x,y II}$\, $) $-(\, $\ref{eqn involving x,y I}$\, $) implies
\begin{eqnarray}\label{eqn involving x,y III}
 x(1-i) & = & y(1+i)
\end{eqnarray}
Again ($\,$\ref{equation I}$\,$) gives,
\begin{eqnarray}\label{eqn involving x,y IV}
 x(1+\gamma_1-i\gamma_2) & = & y(1+\gamma_1+i\gamma_2)
\end{eqnarray}
($\, $\ref{eqn involving x,y I}$\, $) and ($\, $\ref{eqn involving x,y IV}$\, $) together implies $x=-y\, $. Hence from ($\, $\ref{eqn involving x,y III}$\, $) we get $x=y=0$, i,e. $\widehat{f_{\gamma}^*}=0$ for all
$\gamma$, which proves triviality of the intersection.
\end{proof}

\begin{lemma}\label{2nd intersection}
$\{M_a + M_b\circ d + M_c\circ d^*:\mathcal{A}_\varTheta \longrightarrow \mathcal{A}_\varTheta \}\bigcap \{M_f\circ dd^*:\mathcal{A}_\varTheta \longrightarrow \mathcal{A}_\varTheta\}=\{0\}$
\end{lemma}
\begin{proof}
Let $\{e_k:k\in \mathbb{Z}^2\}$ denotes the standard orthonormal basis of $\ell^2(\mathbb{Z}^2)\, $. Any element from the intersection must satisfy
\begin{eqnarray*}
&   & \langle e_\alpha\, ,\, (M_a + M_b\circ d + M_c\circ d^*)(e_\beta)\, \rangle = \langle e_\alpha\, ,\, M_f\circ dd^*(e_\beta)\, \rangle\quad\forall\,\alpha, \beta\in \mathbb{Z}^2\,.\\
& \Rightarrow & \langle\, \sum_k \widehat{a_k^*}e_{k+\alpha}\, ,\, e_\beta)\, \rangle + \langle\, \sum_k \widehat{b_k^*}e_{k+\alpha}\, ,\, e_\beta)\, \rangle(\beta_1-i\beta_2) + \\ 
&    & \langle\, \sum_k \widehat{c_k^*}e_{k+\alpha}\, ,\, e_\beta)\, \rangle(\beta_1+i\beta_2) = \langle\, \sum_k \widehat{f_k^*}e_{k+\alpha}\, ,\, e_\beta\, \rangle(\beta_1^2+\beta_2^2)
\end{eqnarray*}
So,
\begin{center}
 $\widehat{a_{\beta -\alpha}^*} + \widehat{b_{\beta -\alpha}^*}(\beta_1-i\beta_2) + \widehat{c_{\beta -\alpha}^*}(\beta_1+i\beta_2) = \widehat{f_{\beta -\alpha}^*}(\beta_1^2+\beta_2^2)$
\end{center}
for all $\, \alpha=(\alpha_1,\alpha_2),\beta=(\beta_1,\beta_2)\in \mathbb{Z}^2\, $, i,e.
\begin{eqnarray}\label{equation II}
\widehat{a_{\gamma}^*} + \widehat{b_{\gamma}^*}(\alpha_1+\gamma_1-i\alpha_2-i\gamma_2) + \widehat{c_{\gamma}^*}(\alpha_1+\gamma_1+i\alpha_2+i\gamma_2)
\end{eqnarray}
\begin{center}
$=\widehat{f_{\gamma}^*}((\alpha_1+\gamma_1)^2+(\alpha_2+\gamma_2)^2)$
\end{center}
where $\, \beta-\alpha=\gamma\in \mathbb{Z}^2\, $. In order to have nontrivial intersection, equation ($\, $\ref{equation II}$\, $) must have nontrivial solution for all $\alpha,\gamma\in \mathbb{Z}^2\, $. Let $\, \widehat{a_{\gamma}^*}=w$, 
$\widehat{b_{\gamma}^*}=x$, $\widehat{c_{\gamma}^*}=y$, and $\, \widehat{f_{\gamma}^*}=z$. So ($\, $\ref{equation II}$\, $) turns to
\begin{eqnarray}\label{equation III}
w + x(\alpha_1+\gamma_1-i\alpha_2-i\gamma_2) + y(\alpha_1+\gamma_1+i\alpha_2+i\gamma_2)
\end{eqnarray}
\begin{center}
$=z((\alpha_1+\gamma_1)^2+(\alpha_2+\gamma_2)^2)$
\end{center}
>From ($\, $\ref{equation III}$\, $) we get
\begin{eqnarray}\label{eqn involving w,x,y,z I}
 w + x(1+\gamma_1-i-i\gamma_2) + y(1+\gamma_1+i+i\gamma_2) & = & z((1+\gamma_1)^2+(1+\gamma_2)^2)
\end{eqnarray}
\begin{eqnarray}\label{eqn involving w,x,y,z II}
 w + x(2+\gamma_1-2i-i\gamma_2) + y(2+\gamma_1+2i+i\gamma_2) & = & z((2+\gamma_1)^2+(2+\gamma_2)^2)
\end{eqnarray}
($\, $\ref{eqn involving w,x,y,z II}$\, $) - ($\, $\ref{eqn involving w,x,y,z I}$\, $) gives,
\begin{eqnarray}\label{eqn involving w,x,y,z III}
x(1-i) + y(1+i)
\end{eqnarray}
\begin{center}
$=z((2+\gamma_1)^2+(2+\gamma_2)^2-(1+\gamma_1)^2-(1+\gamma_2)^2)$
\end{center}
again, ($\, $\ref{equation III}$\, $) gives
\begin{eqnarray}\label{eqn involving w,x,y,z IV}
w+x(\gamma_1-i\gamma_2) + y(\gamma_1+i\gamma_2) & = & z(\gamma_1^2+\gamma_2^2)
\end{eqnarray}
($\, $\ref{eqn involving w,x,y,z I}$\, $) - ($\, $\ref{eqn involving w,x,y,z IV}$\, $) gives,
\begin{eqnarray}\label{eqn involving w,x,y,z V}
x(1-i) + y(1+i)& = & z((1+\gamma_1)^2+(1+\gamma_2)^2-\gamma_1^2-\gamma_2^2)
\end{eqnarray}
Finally ($\, $\ref{eqn involving w,x,y,z III}$\, $) - ($\, $\ref{eqn involving w,x,y,z V}$\, $) gives $\, z=0$. Hence $\widehat{f_{\gamma}^*}=0\, $ for all $\, \gamma$ i,e. intersection is trivial.
\end{proof}

\begin{lemma}\label{3rd intersection}
$\{M_a\circ d + M_b\circ d^*:\mathcal{A}_\varTheta \longrightarrow \mathcal{A}_\varTheta \}\bigcap \{M_f:\mathcal{A}_\varTheta \longrightarrow \mathcal{A}_\varTheta\}=\{0\}$
\end{lemma}
\begin{proof}
Since $\,d(1)=d^*(1)=0\,$ for $1\in\mathcal{A}_\varTheta$, this follows trivially.

\end{proof}

\begin{proposition}\label{bijection for torus II}
The following map
\begin{eqnarray*}
\Phi:\pi\left(\Omega^2(\widetilde{\mathcal{A}_\varTheta})\right)\longrightarrow \mathcal{A}_\varTheta\otimes \mathbb{C}^8
\end{eqnarray*}
\begin{center}
 $\quad\quad\Phi=(\widetilde\Phi\, ,\, \widetilde\Phi^\prime)$
\end{center}
where
\begin{center}
$\,\,\widetilde\Phi:T_{a,b,c}\longmapsto (ad(b)d^*(c)+abdd^*(c),abd^*(c),ad(bc),abc)\,$,
\end{center}
and
\begin{center}
$\quad\quad\quad\quad\quad\widetilde\Phi^\prime:\widetilde{T_{a^\prime,b^\prime,c^\prime}}\longmapsto (a^\prime d^*(b^\prime)d(c^\prime)+a^\prime b^\prime d^*d(c^\prime),a^\prime b^\prime d(c^\prime),a^\prime d^*(b^\prime
c^\prime),a^\prime b^\prime c^\prime)\,$,
\end{center}
is a linear bijection, where $\,T_{a,b,c}=adbd^*c$ and $\widetilde{T_{a^\prime,b^\prime,c^\prime}}=a^\prime d^*b^\prime dc^\prime$ as define in $\,\ref{definition of Tabc}\,$.
\end{proposition}
\begin{proof}
Since $d(U)=d^*(U)=U$ and $UU^*=U^*U=I$, Lemma $\, $\ref{1st intersection}$\, $, Lemma $\, $\ref{2nd intersection}$\, $ and Lemma $\, $\ref{3rd intersection}$\, $ proves well-definedness as well as injectivity of $\Phi\,$.
To see surjectivity observe that
\begin{eqnarray*}
&   & T_{aU^*,1,U}-T_{aU^*,U,1}-T_{-ia,V^*,V}-T_{ia,1,1}+T_{-ib,V^*,V}-T_{-ib,1,1}\\
&   & +T_{cU^*,U,1}-T_{c,1,1}+T_{e,1,1}\xrightarrow{\widetilde\Phi}(a,b,c,e)\in \mathcal{A}_\varTheta\oplus \mathcal{A}_\varTheta\oplus \mathcal{A}_\varTheta\oplus \mathcal{A}_\varTheta
\end{eqnarray*}
and
\begin{eqnarray*}
&   & T_{a^\prime U^*,1,U}-T_{a^\prime U^*,U,1}-T_{-ia^\prime,V,V^*}-T_{ia^\prime,1,1}+T_{-ib^\prime,V,V^*}-T_{-ib^\prime,1,1}\\
&   & +T_{c^\prime U^*,U,1}-T_{c^\prime,1,1}+T_{e^\prime,1,1}\xrightarrow{\widetilde\Phi^\prime}(a^\prime,b^\prime,c^\prime,e^\prime)\in \mathcal{A}_\varTheta\oplus \mathcal{A}_\varTheta\oplus \mathcal{A}_\varTheta\oplus
 \mathcal{A}_\varTheta
\end{eqnarray*}
This completes the proof.
\end{proof}

\begin{proposition}\label{denominator of 2-form}
$\pi\left(dJ_0^1(\widetilde{\mathcal{A}_\varTheta})\right)\cong \mathcal{A}_\varTheta\otimes \mathbb{C}^6$.
\end{proposition}
\begin{proof}
Elements of $\,\pi\left(dJ_0^1(\widetilde{\mathcal{A}_\varTheta})\right)$ looks like
\begin{center}
$\sum [D,pa+qb][D,pe+qf]\, \, $ such that $\, \, \sum (pa+qb)[D,pe+qf]=0\,,$
\end{center}
where $\,p=(1+\gamma)/2=\begin{pmatrix}
 1 & 0\\
 0 & 0
\end{pmatrix}$ and $\,q=(1-\gamma)/2=\begin{pmatrix}
 0 & 0\\
 0 & 1
\end{pmatrix}$ are the projections onto the eigenspaces of $\gamma\,$. Expanding the commutators and simplifying we get
\begin{eqnarray}\label{elements of kernel I}
\sum \begin{pmatrix}
 -add^*e+aedd^* & 0\\
 0 & bfd^*d-bd^*df
\end{pmatrix}& s.t. & \, \begin{array}{lcl}
                     \sum adf=\sum aed\\
                     \sum bd^*e=\sum bfd^*
                    \end{array}
\end{eqnarray}
Recall that $\,T_{a,b,c}=adbd^*c\,$ and $\,\widetilde{T_{a^\prime,b^\prime,c^\prime}}=a^\prime d^*b^\prime dc^\prime\, $. So equation ($\, $\ref{elements of kernel I}$\, $) becomes,
\begin{eqnarray}\label{elements of kernel II}
\sum \begin{pmatrix}
 T_{ae,1,1}-T_{a,1,e} & 0\\
 0 & \widetilde{T_{bf,1,1}}-\widetilde{T_{b,1,f}}
\end{pmatrix}& s.t. & \, \begin{array}{lcl}
                     \sum adf=\sum aed\\
                     \sum bd^*e=\sum bfd^*
                    \end{array}
\end{eqnarray}
The bijection of Proposition $\, $\ref{bijection for torus II}$\, $ gives,
\begin{center}
$\widetilde{\Phi}\,(T_{ae,1,1}-T_{a,1,e})=(-add^*(e),-ad^*(e),-ad(e),0)$
\end{center}
\begin{center}
$\widetilde{\Phi}^\prime(\widetilde{T_{bf,1,1}}-\widetilde{T_{b,1,f}})=(-bd^*d(f),-bd(f),-bd^*(f),0)$
\end{center}
To fullfil our claim it is enough to show that elements of the form
\begin{eqnarray*}
\textbf(add^*(e),ad^*(e),ad(e),bd^*d(f),bd(f),bd^*(f)\textbf)
\end{eqnarray*}
can generate $\mathcal{A}_\varTheta\otimes \mathbb{C}^6\,$, where conditions in equation ($\, $\ref{elements of kernel I}$\, $) hold. Choose any arbitrary element $\, (a_1,a_2,a_3,a_1^\prime,a_2^\prime,a_3^\prime)
\in\mathcal{A}_\varTheta \otimes\mathbb{C}^6\,$. Observe that,
\begin{eqnarray*}
&   & \textbf{(}a_1V^*dd^*(V),a_1V^*d^*(V),a_1V^*d(V),a_1^\prime V^*d^*d(V),a_1^\prime V^*d(V),a_1^\prime V^*d^*(V)\textbf{)}\\
& +  & \textbf{(}a_1Vdd^*(V^*),a_1Vd^*(V^*),a_1Vd(V^*),a_1^\prime Vd^*d(V^*),a_1^\prime Vd(V^*),a_1^\prime Vd^*(V^*)\textbf{)}\\
& = & (2a_1,0,0,2a_1^\prime,0,0)
\end{eqnarray*}
and conditions of ($\, $\ref{elements of kernel I}$\, $) also satisfied. Hence $(a_1,0,0,a_1^\prime,0,0)\in \pi\left(dJ_0^1(\widetilde{\mathcal{A}_\varTheta})\right)$. Now,
\begin{eqnarray*}
&   & \textbf{(}ia_3Udd^*(U^*),ia_3Ud^*(U^*),ia_3Ud(U^*),0,0,0\textbf{)}\\
& +  & \textbf{(}a_3V^*dd^*(V),a_3V^*d^*(V),a_3V^*d(V),0,0,0\textbf{)}\\
& +  & \textbf{(}-\frac{1}{2}(a_3+ia_3)V^*dd^*(V),-\frac{1}{2}(a_3+ia_3)V^*d^*(V),-\frac{1}{2}(a_3+ia_3)V^*d(V),0,0,0\textbf{)}\\
& +  & \textbf{(}-\frac{1}{2}(a_3+ia_3)Vdd^*(V^*),-\frac{1}{2}(a_3+ia_3)Vd^*(V^*),-\frac{1}{2}(a_3+ia_3)Vd(V^*),0,0,0\textbf{)}\\
& = & (0,0,-2ia_3,0,0,0)
\end{eqnarray*}
and conditions of ($\, $\ref{elements of kernel I}$\, $) also satisfied. Hence, $(0,0,a_3,0,0,0)\in \pi\left(dJ_0^1(\widetilde{\mathcal{A}_\varTheta})\right)$. Finally
\begin{eqnarray*}
&   & \textbf{(}0,0,0,ia_2^\prime Ud^*d(U^*),ia_2^\prime Ud(U^*),ia_2^\prime Ud^*(U^*)\textbf{)}\\
& +  & \textbf{(}0,0,0,a_2^\prime V^*d^*d(V),a_2^\prime V^*d(V),a_2^\prime V^*d^*(V)\textbf{)}\\
& +  & \textbf{(}0,0,0,-\frac{1}{2}(a_2^\prime+ia_2^\prime)V^*d^*d(V),-\frac{1}{2}(a_2^\prime+ia_2^\prime)V^*d(V),-\frac{1}{2}(a_2^\prime+ia_2^\prime)V^*d^*(V)\textbf{)}\\
& +  & \textbf{(}0,0,0,-\frac{1}{2}(a_2^\prime+ia_2^\prime)Vd^*d(V^*),-\frac{1}{2}(a_2^\prime+ia_2^\prime)Vd(V^*),-\frac{1}{2}(a_2^\prime+ia_2^\prime)Vd^*(V^*)\textbf{)}\\
& = & (0,0,0,0,-2ia_2^\prime,0) 
\end{eqnarray*}
and conditions of ($\, $\ref{elements of kernel I}$\, $) also satisfied. Hence, $(0,0,0,0,a_2^\prime,0)\in \pi\left(dJ_0^1(\widetilde{\mathcal{A}_\varTheta})\right)$. Thus combining we have $(0,0,a_3,0,a_2^\prime,0)\in\pi
\left(dJ_0^1(\widetilde{\mathcal{A}_\varTheta})\right)$. Similarly one can show that $(0,a_2,0,0,0,a_3^\prime)\in \pi\left(dJ_0^1(\widetilde{\mathcal{A}_\varTheta})\right)$ and this completes the proof.
\end{proof}

\begin{proposition}\label{bimodule action for torus II}
The following action
\begin{center}
$\begin{pmatrix}
 x & 0\\
 0 & y
\end{pmatrix}.(a_1,a_2) := (xa_1,ya_2)$ 
\end{center}
\begin{center}
$(a_1,a_2).\begin{pmatrix}
 x & 0\\
 0 & y
\end{pmatrix}\,:= (a_1y,a_2x)$. 
\end{center}
defines an $\, \widetilde{\mathcal{A}_\varTheta}$-bimodule structure on $\frac{\pi\left(\Omega^2(\widetilde{\mathcal{A}_\varTheta})\right)}{\pi\left(dJ_0^1(\widetilde{\mathcal{A}_\varTheta})\right)}\cong\mathcal{A}_\varTheta
\oplus \mathcal{A}_\varTheta\,$.
\end{proposition}
\begin{proof}
If we define
\begin{center}
$\begin{pmatrix}
 x & 0\\
 0 & y
\end{pmatrix}.(a_1,a_2) := \Phi\left(\begin{pmatrix}
 x & 0\\
 0 & y
\end{pmatrix}.\Phi^{-1}(a_1,a_2)\right)$ 
\end{center}
\begin{center}
$(a_1,a_2).\begin{pmatrix}
 x & 0\\
 0 & y
\end{pmatrix} := \Phi\left(\Phi^{-1}(a_1,a_2).\begin{pmatrix}
 x & 0\\
 0 & y
\end{pmatrix}\right)$ 
\end{center}
where $\,\Phi\,$ is as defined in Proposition $\, $\ref{bijection for torus II}$\, $, then clearly it is a bimodule action induced by that on $\Omega_D^2(\widetilde{\mathcal{A}_\varTheta})$. One can verify that these actions match with 
the ones defined in question.
\end{proof}

\begin{theorem}\label{forms for torus}
For noncommutative torus we have,
\begin{enumerate}
 \item $\Omega_D^1(\widetilde{\mathcal{A}_\varTheta})\cong \mathcal{A}_\varTheta \bigoplus \mathcal{A}_\varTheta \bigoplus \mathcal{A}_\varTheta \bigoplus \mathcal{A}_\varTheta\,\,$, as $\widetilde{\mathcal{A}_\varTheta}$-bimodule.
 \item $\Omega_D^n(\widetilde{\mathcal{A}_\varTheta})\cong \mathcal{A}_\varTheta \bigoplus \mathcal{A}_\varTheta\,\,$, for all $n\geq 2$ as $\widetilde{\mathcal{A}_\varTheta}$-bimodule.
\end{enumerate}
\end{theorem}
\begin{proof}
Proposition $\,$\ref{space of 1-form for torus}$\,$ gives part $(1)$. Proposition $\,$\ref{bijection for torus II}$\,$ and $\,$\ref{denominator of 2-form}$\,$ proves part $(2)$ for $n=2$. The fact that the isomorphisms in
Propositions $\,$\ref{bijection for torus II}$\,$, $\,$\ref{denominator of 2-form}$\,$ are not only $\mathbb{C}\,$-linear but also $\widetilde{\mathcal{A}_\varTheta}$-bimodule isomorphisms follows from the defining property
of the bimodule action in Proposition $\,$\ref{bimodule action for torus II}$\,$.

We need to prove part $(2)$ for $n\geq 3$. For that purpose first note that
\begin{eqnarray}\label{mult. of various matrices}
\begin{cases}
 \begin{pmatrix}
   0 & 0\\
   0 & 1
  \end{pmatrix}\begin{pmatrix}
   0 & 1\\
   0 & 0
  \end{pmatrix}=\begin{pmatrix}
   1 & 0\\
   0 & 0
  \end{pmatrix}\begin{pmatrix}
   0 & 0\\
   0 & 1
  \end{pmatrix}=\begin{pmatrix}
   0 & 1\\
   0 & 0
  \end{pmatrix}^2=\begin{pmatrix}
   0 & 0\\
   1 & 0
  \end{pmatrix}^2=\begin{pmatrix}
   0 & 0\\
   0 & 0
  \end{pmatrix}\\
 \begin{pmatrix}
   0 & 0\\
   0 & 1
  \end{pmatrix}\begin{pmatrix}
   1 & 0\\
   0 & 0
  \end{pmatrix}=\begin{pmatrix}
   0 & 1\\
   0 & 0
  \end{pmatrix}\begin{pmatrix}
   1 & 0\\
   0 & 0
  \end{pmatrix}=\begin{pmatrix}
   0 & 0\\
   1 & 0
  \end{pmatrix}\begin{pmatrix}
   0 & 0\\
   0 & 1
  \end{pmatrix}=\begin{pmatrix}
   0 & 0\\
   0 & 0
  \end{pmatrix}
\end{cases}
\end{eqnarray} These matrices are the key role to compute $\Omega_D^n(\widetilde{\mathcal{A}_\varTheta})$ for all $n\geq 3$. Now for any unital algebra $\mathcal{A}\,$,
\begin{eqnarray}\label{tensor product of forms}
\Omega^n(\mathcal{A})=\underbrace{\Omega^1(\mathcal{A})\otimes_{\mathcal{A}}\ldots\ldots\otimes_{\mathcal{A}}\Omega^1(\mathcal{A})}_{n\,\,times}\,.
\end{eqnarray}
By Lemma $\,$\ref{bijection for torus}$\,$ we have $$\pi\left(\Omega^1(\widetilde{\mathcal{A}_\varTheta})\right)=(\mathcal{A}_\varTheta\oplus\mathcal{A}_\varTheta)\otimes_{\mathbb{C}}\begin{pmatrix}
   0 & 1\\
   0 & 0
  \end{pmatrix}+(\mathcal{A}_\varTheta\oplus\mathcal{A}_\varTheta)\otimes_{\mathbb{C}}\begin{pmatrix}
   0 & 0\\
   1 & 0
  \end{pmatrix}\,.$$ In view of Proposition $\,$\ref{bimodule action for torus I}$\,$ and using $(\,$\ref{mult. of various matrices}$\,)$, $(\,$\ref{tensor product of forms}$\,)$ we get $\,\pi\left(\Omega^n(\widetilde
{\mathcal{A}_\varTheta})\right)\cong\mathcal{A}_\varTheta\otimes\mathbb{C}^{2n}\bigoplus\mathcal{A}_\varTheta\otimes\mathbb{C}^{2n}\,$ for all $\,n\geq 3$ (actually true for all $\,n\geq 1$ by part $(1)$ and Proposition
$\,$\ref{bijection for torus II}$\,$). We will show the following $$\pi\left(dJ_0^n(\widetilde{\mathcal{A}_\varTheta})\right)\cong\mathcal{A}_\varTheta\otimes\mathbb{C}^{2n+1}\bigoplus\mathcal{A}_\varTheta\otimes
\mathbb{C}^{2n+1}\quad\forall\,n\geq 2\,.$$ Recall from Lemma $\,$\ref{[D^2,. ] (1)}$\,$, $[D^2,a]\in\pi(dJ_0^1)\,$. It is then easy to prove that $$\pi(dJ_0^n)=\sum_{i=0}^{n-1}\pi\left(\Omega^i\otimes_{\mathcal{A}}
J^2\otimes_{\mathcal{A}}\Omega^{n-1-i}\right)\quad for\,\,all\,\,n\geq 2$$ by writing down any arbitrary element of $\pi(dJ_0^n)$ and then passing $D$ through the commutators from left to right. Hence for all $\,n\geq 2$ odd,
\begin{eqnarray*}
 \pi\left(dJ_0^n(\widetilde{\mathcal{A}_\varTheta})\right) & = & \sum_{i=0\,,\,i\,even}^{n-1}\pi\left(\Omega^i\otimes J^2\otimes\Omega^{n-1-i}\right)+\sum_{i=1\,,\,i\,odd}^{n-1}\pi\left(\Omega^i\otimes J^2\otimes\Omega^{n-1-i}\right)\\
                                                           & = & \sum_{i=0\,,\,i\,even}^{n-1}\left(\mathcal{A}_\varTheta\otimes\mathbb{C}^{2i}\otimes\begin{pmatrix}
   1 & 0\\
   0 & 0
  \end{pmatrix}+\mathcal{A}_\varTheta\otimes\mathbb{C}^{2i}\otimes\begin{pmatrix}
   0 & 0\\
   0 & 1
  \end{pmatrix}\right)\\
&   & \quad\quad\quad\quad\left(\mathcal{A}_\varTheta\otimes\mathbb{C}^{3}\otimes\begin{pmatrix}
   1 & 0\\
   0 & 0
  \end{pmatrix}+\mathcal{A}_\varTheta\otimes\mathbb{C}^{3}\otimes\begin{pmatrix}
   0 & 0\\
   0 & 1
  \end{pmatrix}\right)\\
&   & \quad\quad\quad\quad\left(\mathcal{A}_\varTheta\otimes\mathbb{C}^{2(n-1-i)}\otimes\begin{pmatrix}
   1 & 0\\
   0 & 0
  \end{pmatrix}+\mathcal{A}_\varTheta\otimes\mathbb{C}^{2(n-1-i)}\otimes\begin{pmatrix}
   0 & 0\\
   0 & 1
  \end{pmatrix}\right)\\
& + & \sum_{i=1\,,\,i\,odd}^{n-1}\,\left(\mathcal{A}_\varTheta\otimes\mathbb{C}^{2i}\otimes\begin{pmatrix}
   0 & 1\\
   0 & 0
  \end{pmatrix}+\mathcal{A}_\varTheta\otimes\mathbb{C}^{2i}\otimes\begin{pmatrix}
   0 & 0\\
   1 & 0
  \end{pmatrix}\right)\\
&   & \quad\quad\quad\quad\left(\mathcal{A}_\varTheta\otimes\mathbb{C}^{3}\otimes\begin{pmatrix}
   1 & 0\\
   0 & 0
  \end{pmatrix}+\mathcal{A}_\varTheta\otimes\mathbb{C}^{3}\otimes\begin{pmatrix}
   0 & 0\\
   0 & 1
  \end{pmatrix}\right)\\
&   & \quad\quad\quad\quad\left(\mathcal{A}_\varTheta\otimes\mathbb{C}^{2(n-1-i)}\otimes\begin{pmatrix}
   0 & 1\\
   0 & 0
  \end{pmatrix}+\mathcal{A}_\varTheta\otimes\mathbb{C}^{2(n-1-i)}\otimes\begin{pmatrix}
   0 & 0\\
   1 & 0
  \end{pmatrix}\right)\\
& = & \mathcal{A}_\varTheta\otimes\mathbb{C}^{2n+1}\otimes\begin{pmatrix}
   1 & 0\\
   0 & 0
  \end{pmatrix}+\mathcal{A}_\varTheta\otimes\mathbb{C}^{2n+1}\otimes\begin{pmatrix}
   0 & 0\\
   0 & 1
  \end{pmatrix}\\
& \cong & \mathcal{A}_\varTheta\otimes\mathbb{C}^{2n+1}\bigoplus\mathcal{A}_\varTheta\otimes\mathbb{C}^{2n+1}\,.
\end{eqnarray*}
Here the last equality uses $(\,$\ref{mult. of various matrices}$\,)$ frequently. Similarly one can do for all $n\geq 2$ even. Hence, for all $n\geq 3$ we have $\,\Omega_D^n(\widetilde{\mathcal{A}_\varTheta})\cong
\mathcal{A}_\varTheta\otimes\mathbb{C}^2\,$ as $\widetilde{\mathcal{A}_\varTheta}$-bimodule where the bimodule action on $\,\mathcal{A}_\varTheta\otimes\mathbb{C}^2\,$ will be specified by Proposition
$\,$\ref{bimodule action for torus II}$\,$.
\end{proof}

\begin{remark}
One can also consider
\begin{center}
 $\begin{array}{lcl}
   D=\begin{pmatrix}
 0 & d\\
 d & 0
\end{pmatrix}
  \end{array}
\, \, \, or \, \, \, \, D=\begin{pmatrix}
 0 & d^*\\
 d^* & 0
\end{pmatrix}$.
\end{center}
However, in that case one will get same answer as in Theorem $\, \ref{forms for torus}\, $. Since in noncommutative geometry it is customary to take
$D=\begin{pmatrix}
 0 & d\\
 d^* & 0
\end{pmatrix}\,$, we provide computation with this value for $D$.
\end{remark}
\medskip

\textbf{Notation~:} $\widetilde{\Omega_D^1}:=\, \mathcal{A}_\varTheta\otimes \mathbb{C}^4\, $ and $\, \widetilde{\Omega_D^2}:=\, \mathcal{A}_\varTheta\otimes \mathbb{C}^2\,$ untill the rest of this section.
\medskip

Now we want to show that $\,\Omega_D^\bullet(\widetilde{\mathcal{A}_\varTheta})$ is cohomologically not trivial. For that purpose we use the isomorphism in Theorem $\,$\ref{forms for torus}$\,$ to compute the differentials
on $\,\widetilde{\Omega_D^1}$ and $\,\widetilde{\Omega_D^2}$.

\begin{proposition}\label{the differential for torus}
The maps
\begin{center}
$\delta:\widetilde{\mathcal{A}_\varTheta}\,\longrightarrow \widetilde{\Omega_D^1}$
\end{center}
\begin{center}
$\quad\quad\quad\quad\quad\quad\quad\quad\begin{pmatrix}
 a & 0\\
 0 & b
\end{pmatrix}\longmapsto (d(b),b-a,d^*(a),a-b)$
\end{center}
and
\begin{center}
$\delta:\widetilde{\Omega_D^1}\,\longrightarrow \widetilde{\Omega_D^2}$
\end{center}
\begin{center}
$\quad\quad\quad(a,b,c,e)\longmapsto (b+e,b+e)$
\end{center}
make the following diagrams
\begin{center}$
\begin{array}{lcl}
\begin{tikzpicture}[node distance=3cm,auto]
\node (Up)[label=above:$\widetilde{d}$]{};
\node (A)[node distance=1cm,left of=Up]{$\widetilde{\mathcal{A}_\varTheta}$};
\node (B)[node distance=1.5cm,right of=Up]{$\Omega^1_D(\widetilde{\mathcal{A}_\varTheta})$};
\node (Down)[node distance=1.5cm,below of=Up, label=below:$\delta$]{};
\node(C)[node distance=1cm,left of=Down]{$\widetilde{\mathcal{A}_\varTheta}$};
\node(D)[node distance=1.5cm,right of=Down]{$\widetilde{\Omega^1_D}$};
\draw[->](A) to (B);
\draw[->](C) to (D);
\draw[->](B)to node{{ $\cong$}}(D);
\draw[->](A)to node[swap]{{ $id$}}(C);
\end{tikzpicture}
\quad\quad\quad\quad
\begin{tikzpicture}[node distance=3cm,auto]
\node (Up)[label=above:$\widetilde{d}$]{};
\node (A)[node distance=1.5cm,left of=Up]{$\Omega^1_D(\widetilde{\mathcal{A}_\varTheta})$};
\node (B)[node distance=1.5cm,right of=Up]{$\Omega^2_D(\widetilde{\mathcal{A}_\varTheta})$};
\node (Down)[node distance=1.5cm,below of=Up, label=below:$\delta$]{};
\node(C)[node distance=1.5cm,left of=Down]{$\widetilde{\Omega^1_D}$};
\node(D)[node distance=1.5cm,right of=Down]{$\widetilde{\Omega^2_D}$};
\draw[->](A) to (B);
\draw[->](C) to (D);
\draw[->](B)to node{{ $\cong$}}(D);
\draw[->](A)to node[swap]{{ $\cong$}}(C);
\end{tikzpicture}
\end{array}$
\end{center}
commutative, where $\,\widetilde{d}:\Omega^\bullet_D(\widetilde{\mathcal{A}_\varTheta})\longrightarrow\Omega^{\bullet+1}_D(\widetilde{\mathcal{A}_\varTheta})$ denotes the differential of Connes complex.
\end{proposition}
\begin{proof}
Use Lemma $\,$\ref{bijection for torus}$\,$ to see commutativity of the first diagram. For the second, take any $(a,b,c,e)\in \widetilde{\Omega_D^1}$ and use $\Phi^{-1}$ of Lemma $\,$\ref{bijection for torus}$\,$ to get an
element in $\pi\left(\Omega^1(\widetilde{\mathcal{A}_\varTheta})\right)$, which is
\begin{eqnarray}\label{element of 1-form}
\begin{pmatrix}
 0 & aU^*dU+bd1-ad1\\
 cU^*d^*U+ed^*1-cd^*1 & 0
\end{pmatrix}.
\end{eqnarray}
Use the fact
\begin{center}$
 \begin{pmatrix}
 -U^* & 0\\
 0 & 0
\end{pmatrix}\left[D,\begin{pmatrix}
 U & 0\\
 0 & 0
\end{pmatrix}\right]+\begin{pmatrix}
 0 & 0\\
 0 & -U^*
\end{pmatrix}\left[D,\begin{pmatrix}
 0 & 0\\
 0 & U
\end{pmatrix}\right]=\begin{pmatrix}
 0 & d\\
 d^* & 0
\end{pmatrix}
$\end{center}
to observe that ($\, $\ref{element of 1-form}$\, $) can be re-written as
\begin{center}
$\begin{pmatrix}
 -aU^*U^* & 0\\
 0 & 0
\end{pmatrix}\left[D,\begin{pmatrix}
 U^2 & 0\\
 0 & 0
\end{pmatrix}\right]+\begin{pmatrix}
 aU^* & 0\\
 0 & 0
\end{pmatrix}\left[D,\begin{pmatrix}
 U & 0\\
 0 & U
\end{pmatrix}\right]+\begin{pmatrix}
 0 & 0\\
 0 & -cU^*U^*
\end{pmatrix}\left[D,\begin{pmatrix}
 0 & 0\\
 0 & U^2
\end{pmatrix}\right]+
\begin{pmatrix}
 0 & 0\\
 0 & cU^*
\end{pmatrix}\left[D,\begin{pmatrix}
 U & 0\\
 0 & U
\end{pmatrix}\right]+$
\end{center}
\begin{center}$
\begin{pmatrix}
 (a-b)U^* & 0\\
 0 & 0
\end{pmatrix}\left[D,\begin{pmatrix}
 U & 0\\
 0 & 0
\end{pmatrix}\right]+\begin{pmatrix}
 0 & 0\\
 0 & (c-e)U^*
\end{pmatrix}\left[D,\begin{pmatrix}
 0 & 0\\
 0 & U
\end{pmatrix}\right]$.
\end{center}
Applying $\,\Phi\circ\widetilde{d}\,$ ($\Phi$ of Proposition $\, $\ref{bijection for torus II}$\, $ and $\widetilde{d}:\Omega_D^1\longrightarrow \Omega_D^2$), we get the following element 
\begin{eqnarray*}
 &   & \textbf{(}d(c)+2c+2a+b\,,\,b+c\,,\,d(e)+a+b\,,\,b+e\,,\,d^*(a)+2a+2c+e\,,\,e+a\,,\\
&   & \,\,d^*(b)+c+e\,,\,\,e+b\textbf{)}\,+\,\pi\left(dJ_0^1(\widetilde{\mathcal{A}_\varTheta})\right)
\end{eqnarray*}
of $\,\Omega_D^2(\widetilde{\mathcal{A}_\varTheta})$. This element is equal to $\,(b+e\,,\,b+e)\in\Omega_D^2(\widetilde{\mathcal{A}_\varTheta})$ by Theorem $\, $\ref{forms for torus}$\, $.
\end{proof}

\begin{remark}
Notice that $\,\delta=\Phi\circ\widetilde{d}\circ\Phi^{-1}\, $, and hence $\,\delta^2=0$. 
\end{remark}

Before we proceed to show that $\,\Omega_D^\bullet(\widetilde{\mathcal{A}_\varTheta})$ is cohomologically not trivial we first compute the cohomologies for $(\Omega_D^\bullet(\mathcal{A}_\varTheta)\, ,d)$, in order to notice
the similarity. To do so recall Proposition $13$, in the last chapter of $\, $(\cite{3})$\, $.
\begin{proposition}[\cite{3}]
For noncommutative torus $\mathcal{A}_\varTheta$, we have
\begin{enumerate}
 \item $\Omega_D^1(\mathcal{A}_\varTheta)\cong\mathcal{A}_\varTheta\oplus\mathcal{A}_\varTheta\,$,
 \item $\Omega_D^2(\mathcal{A}_\varTheta)\cong\mathcal{A}_\varTheta\,$,
 \item The differentials $\,\widetilde{d}:\mathcal{A}_\varTheta\longrightarrow \Omega_D^1(\mathcal{A}_\varTheta)$ and $\,\widetilde{d}:\Omega_D^1(\mathcal{A}_\varTheta)\longrightarrow \Omega_D^2(\mathcal{A}_\varTheta)$ are given by
\begin{center}
 $\,\,\widetilde{d}:a\,\longmapsto (\delta_1a\, ,\, \delta_2a)$
\end{center}
\begin{center}
 $\widetilde{d}:(a_1,a_2)\longmapsto \delta_2(a_1)-\delta_1(a_2)$
\end{center}
\end{enumerate}
\end{proposition}

\begin{remark}
For $n\geq 3$, the space of higher forms $\,\Omega_D^n(\mathcal{A}_\varTheta)$ vanish.
To see this first observe that $[D,a]=\delta_1(a)\otimes\sigma_1+\delta_2(a)\otimes\sigma_2$ where $\sigma_1, \sigma_2$ are the spin matrices satisfing $\sigma_i\sigma_j+\sigma_j\sigma_i=2\delta_{ij}$.
The isomorphism $\,\Omega_D^1(\mathcal{A}_\varTheta)\cong\mathcal{A}_\varTheta\otimes\mathbb{C}^2\,$ is obtained using the linear independence of $\sigma_1, \sigma_2$ in $M_2(\mathbb{C})$. The isomorphism
$\,\pi(\Omega^2(\mathcal{A}_\varTheta))\cong\mathcal{A}_\varTheta\otimes\mathbb{C}^2\,$ is obtained using the linear independence of $I_2$ and $\sigma_1\sigma_2$ in $M_2(\mathbb{C})$ and in this way one will obtain that
$\,\pi(dJ_0^1(\mathcal{A}_\varTheta))\cong\mathcal{A}_\varTheta\otimes I_2$. Because of this reason $\,\Omega_D^2(\mathcal{A}_\varTheta)\cong\mathcal{A}_\varTheta\,$. Observe that $\,\pi(\Omega^3(\mathcal{A}_\varTheta))=
\mathcal{A}_\varTheta\otimes\sigma_1+\mathcal{A}_\varTheta\otimes\sigma_2$ and hence $\,\pi(\Omega^3(\mathcal{A}_\varTheta))\cong\mathcal{A}_\varTheta\otimes\mathbb{C}^2\,$. Now recall that $J^\bullet$ is a graded ideal in
$\,\Omega^\bullet$ and hence we have $$\pi\left(\Omega^1(\mathcal{A}_\varTheta)J^2(\mathcal{A}_\varTheta)\right)\subseteq\pi\left(J^3(\mathcal{A}_\varTheta)\right)\subseteq\pi\left(\Omega^3(\mathcal{A}_\varTheta)\right)\,.$$
This shows that $\,\pi\left(J^3(\mathcal{A}_\varTheta)\right)=\pi\left(dJ_0^2(\mathcal{A}_\varTheta)\right)\cong\mathcal{A}_\varTheta\otimes\mathbb{C}^2\,$ i,e. $\,\Omega_D^3(\mathcal{A}_\varTheta)=\{0\}$. Now
note that $\,\Omega^n(\mathcal{A})=\underbrace{\Omega^1(\mathcal{A})\otimes_{\mathcal{A}}\ldots\ldots\otimes_{\mathcal{A}}\Omega^1(\mathcal{A})}_{n\,\,times}\,$ for any unital algebra $\mathcal{A}\,$. Hence,
$\pi\left(\Omega^n(\mathcal{A}_\varTheta)\right)\cong\mathcal{A}_\varTheta\otimes\mathbb{C}^2\,$ for all $n\geq 4$. Finally, the inclusion $$\pi\left(\Omega^{n-2}(\mathcal{A}_\varTheta)J^2(\mathcal{A}_\varTheta)\right)
\subseteq\pi\left(J^n(\mathcal{A}_\varTheta)\right)\subseteq\pi\left(\Omega^n(\mathcal{A}_\varTheta)\right)$$ proves that $\,\Omega_D^n(\mathcal{A}_\varTheta)=\{0\}$ for all $n\geq 4$. This is needed in the next Lemma.
\end{remark}

\begin{lemma}
 The cohomologies $H^\bullet(\mathcal{A}_\varTheta)$ are given by~,
\begin{enumerate}
 \item $H^0(\mathcal{A}_\varTheta)\cong \mathbb{C}\, $,
 \item $H^1(\mathcal{A}_\varTheta)\cong \mathbb{C}\oplus \mathbb{C}\, $,
 \item $H^2(\mathcal{A}_\varTheta)\cong \mathbb{C}\, $.
\end{enumerate}
\end{lemma}
\begin{proof}
\begin{enumerate}
 \item We have
  \begin{eqnarray*}
   H^0(\mathcal{A}_\varTheta) & = & \{a\in \mathcal{A}_\varTheta:\delta_1(a)=\delta_2(a)=0\}\\
                              & \cong & \mathbb{C} 
  \end{eqnarray*}
 \item We have
  \begin{eqnarray*}
   H^1(\mathcal{A}_\varTheta) & = & \frac{\{(a,b):a,b\in \mathcal{A}_\varTheta ; \delta_2(a)=\delta_1(b)\}}{\{(\delta_1(a),\delta_2(a)):a\in \mathcal{A}_\varTheta\}}            
  \end{eqnarray*}
  Let
   \begin{center}
    $a=\sum_{m,n}\alpha_{m,n}U^mV^n-\alpha_{0,0}\, \, ,\, \, b=\sum_{p,q}\beta_{p,q}U^pV^q-\beta_{0,0}$
   \end{center}
    i,e. $a,b\notin \mathbb{C}1$. Then $\delta_2(a)=\delta_1(b)$ will imply
   \begin{eqnarray*}
    \sum_{m\neq 0,n\neq 0}n\alpha_{m,n}U^mV^n+\sum_{n\neq 0}n\alpha_{0,n}V^n & = & \sum_{p\neq 0,q\neq 0}p\beta_{p,q}U^pV^q+\sum_{p\neq 0}p\beta_{p,0}U^p
   \end{eqnarray*}
   If $\{e_{mn}\}_{m,n\in \mathbb{Z}}$ be orthonormal basis of $\ell^2(\mathbb{Z}^2)$ then we get
   \begin{center}
    $\beta_{p,0}=0\, \, \forall\, p\neq 0\, $ ; $\, \alpha_{0n}=0\, \, \forall\, n\neq 0\, $ ; $\, n\alpha_{mn}=m\beta_{mn}\, \, \forall\, m\neq 0,n\neq 0$.
   \end{center}
   Let
   \begin{eqnarray*}
    c & = & \sum_{m\neq 0,n\neq 0}\frac{\gamma_{m,n}}{mn}U^mV^n+\sum_{m\neq 0}\frac{\alpha_{m,0}}{m}U^m+\sum_{n\neq 0}\frac{\beta_{0,n}}{n}V^n
   \end{eqnarray*}
   For $m\neq 0,n\neq 0$, if we choose $\gamma_{m,n}=n\alpha_{mn}$ then we get $\delta_1(c)=a$ and $\delta_2(c)=b$ which proves our claim.
 \item Finally,
  \begin{eqnarray*}
   H^2(\mathcal{A}_\varTheta) & = & \frac{\mathcal{A}_\varTheta}{\{\delta_2(a)-\delta_1(b):a,b\in \mathcal{A}_\varTheta\}}
  \end{eqnarray*}
  Let $\,a\in \mathcal{A}_\varTheta$ be s.t. $a\notin \mathbb{C}1$. Let $\,a=\sum_{m\neq 0\,\, or\,\, n\neq 0}\,\alpha_{m,n}U^mV^n$. Then
  \begin{center}
   $a=\sum_{m\in \mathbb{Z},n\in \mathbb{Z}-\{0\}}\alpha_{m,n}U^mV^n + \sum_{m\in \mathbb{Z}-\{0\}}\alpha_{m,0}U^m$
  \end{center}
  Consider $\,b=\sum_{m\in \mathbb{Z},n\in \mathbb{Z}-\{0\}}\frac{\alpha_{m,n}}{n}U^mV^n\,$ and $\,c=-\sum_{m\in \mathbb{Z}-\{0\}}\frac{\alpha_{m,0}}{m}U^m$. Then $\delta_2(b)-\delta_1(c)=a\,$, which proves our claim.
\end{enumerate}
\end{proof}

\begin{theorem}
If $\,\widetilde{H^\bullet(\mathcal{A}_\varTheta)}$ denotes the cohomologies of the chain complex $\,\left(\Omega_D^\bullet(\widetilde{\mathcal{A}_\varTheta})\,,\delta\right)\,$, then we have
\begin{enumerate}
 \item $\widetilde{H^0(\mathcal{A}_\varTheta)}\cong \mathbb{C}\,$,
 \item $\widetilde{H^1(\mathcal{A}_\varTheta)}\cong \mathbb{C}\oplus \mathbb{C}\oplus \mathcal{A}_\varTheta/\mathbb{C}\,$.
\end{enumerate}
\end{theorem}
\begin{proof}
\begin{enumerate}
 \item We have
  \begin{eqnarray*}
    \widetilde{H^0(\mathcal{A}_\varTheta)} & = & \{\begin{pmatrix}
    a & 0\\
    0 & b
    \end{pmatrix}: d(b)=0,d^*(a)=0,a=b\}\\
     & = & \{\begin{pmatrix}
    a & 0\\
    0 & a
    \end{pmatrix}: \delta_1(a)=\delta_2(a)=0\}\\
    & \cong & \mathbb{C}
   \end{eqnarray*}
 \item We have
  \begin{eqnarray*}
    \widetilde{H^1(\mathcal{A}_\varTheta)} & = & \frac{\{(a,b,c,e):b+e=0\}}{\{(d(f),f-g,d^*(g),g-f)\}}
  \end{eqnarray*}
   Let $\mathcal{M}=\{(a,b,c,-b):a,b,c\in \mathcal{A}_\varTheta\}$ and $\mathcal{N}=\{(d(f),f-g,d^*(g),g-f):,f,g\in \mathcal{A}_\varTheta\}$. Clearly $\mathcal{M}\cong \mathcal{A}_\varTheta\oplus 
   \mathcal{A}_\varTheta\oplus \mathcal{A}_\varTheta$. Now define
   \begin{center}
    $\psi:\mathcal{N}\bigoplus \mathbb{C}\bigoplus \mathcal{A}_\varTheta/\mathbb{C}\bigoplus \mathbb{C}\longrightarrow \mathcal{M}$
   \end{center}
   \begin{center}
    $(d(f),f-g,d^*(g),g-f,\lambda_1,a,\lambda_2)\longmapsto (d(f)+\lambda_1,f-g+a,d^*(g)+\lambda_2)$
   \end{center}
   This map is $\mathbb{C}$-linear and one-one. To see surjectivity take any $(a,b,c)\in \mathcal{A}_\varTheta^3$. Suppose $a=\sum \alpha_{m,n}U^mV^n$. If we choose 
   \begin{center}
    $f=\sum_{m\neq 0\, or\, n\neq 0} \frac{1}{m-in}\alpha_{m,n}U^mV^n$
   \end{center}
   then $d(f)=a-\alpha_{0,0}$ and we see that
   \begin{center}
    $(d(f),f,0,-f,\alpha_{0,0},-f,0)\longmapsto (a,0,0)$
   \end{center}
   Now suppose $b=\sum \beta_{m,n}U^mV^n$. If we choose $f=\beta_{0,0},g=0$ then 
   \begin{center}
    $(0,\beta_{0,0},0,-\beta_{0,0},0,b-\beta_{0,0},0)\longmapsto (0,b,0)$
   \end{center}
   Finally let $c=\sum \gamma_{m,n}U^mV^n$ and choose
   \begin{center}
    $g=\sum_{m\neq 0\, or\, n\neq 0} \frac{1}{m+in}\gamma_{m,n}U^mV^n$
   \end{center}
   then we see that
   \begin{center}
    $(0,-g,c-\gamma_{0,0},g,0,g,\gamma_{0,0})\longmapsto (0,0,c)$
   \end{center}
   This shows that $\psi$ is a linear isomorphism with $\psi(\mathcal{N})=\mathcal{N}$ and hence our claim has been justified.
\end{enumerate}
This shows that the complex $\,\Omega_D^\bullet(\widetilde{\mathcal{A}_\varTheta})\,$ is cohomologically not trivial.
\end{proof}

\end{document}